\documentclass[review,a4page]{elsarticle}

\usepackage{float}
\usepackage{pgf,tikz}
\usepackage{subcaption}
\usetikzlibrary{positioning}
\usetikzlibrary{arrows,patterns}
\usetikzlibrary{decorations.pathreplacing}
\usepackage[utf8]{inputenc}
\usepackage{graphicx}
\usepackage{srcltx}
\usepackage{eurosym}
\usepackage{mathtools}
\allowdisplaybreaks
\usepackage{enumerate}
\usepackage{amsmath}
\usepackage{amsfonts}
\usepackage{amssymb}
\usepackage{amsthm}
\usepackage{graphicx}
\usepackage{mathrsfs}
\usepackage{xcolor}
\usepackage{exscale}
\usepackage{latexsym}
\usepackage{esvect}
\usepackage[T1]{fontenc}

\usepackage{calc}
\usepackage[titletoc,toc,page]{appendix}

\usepackage[colorlinks,plainpages=true,pdfpagelabels,hypertexnames=true,colorlinks=true,pdfstartview=FitV,linkcolor=blue,citecolor=green,urlcolor=black,pagebackref]{hyperref}
\PassOptionsToPackage{unicode}{hyperref}
\PassOptionsToPackage{naturalnames}{hyperref}
\usepackage{url}
\usepackage{enumerate}
\usepackage[shortlabels]{enumitem}
\usepackage{bookmark}
\usepackage{wasysym}
\usepackage{esint}
\usepackage[titletoc,toc,page]{appendix}
\newcommand*\bb{\mathbb}

\newcommand *\w{^\wedge}

\newcommand{\pa}{\partial}

\makeatletter
\newcommand{\vo}{\vec{o}\@ifnextchar{^}{\,}{}}
\makeatother

\def\YYint#1#2#3{{\setbox0=\hbox{$#1{#2#3}{\iint}$}
    \vcenter{\hbox{$#2#3$}}\kern-.50\wd0}}

\def\XXint#1#2#3{{\setbox0=\hbox{$#1{#2#3}{\int}$}
    \vcenter{\hbox{$#2#3$}}\kern-.50\wd0}}

\makeatletter
\def\namedlabel#1#2{\begingroup
   \def\@currentlabel{#2}%
   \label{#1}\endgroup
}
\makeatother
\makeatletter
\newcommand{\rmh}[1]{\mathpalette{\raisem@th{#1}}}
\newcommand{\raisem@th}[3]{\hspace*{-1pt}\raisebox{#1}{$#2#3$}}
\makeatother

\newcommand{\redref}[2]{\texorpdfstring{\protect\hyperlink{#1}{\textcolor{black}{(}\textcolor{red}{#2}\textcolor{black}{)}}}{}}
\newcommand{\redlabel}[2]{\hypertarget{#1}{\textcolor{black}{(}\textcolor{red}{#2}\textcolor{black}{)}}}

\newcommand{\descitem}[2]{\item[{#1:}]\label{#2}}
\newcommand{\descref}[2]{\hyperref[#1]{\textnormal{\textcolor{black}{(}\textcolor{blue}{\bf #2}\textcolor{black}{)}}}}

\newcommand{\dref}[2]{\hyperref[#1]{\textcolor{black}{(}\textcolor{blue}{\bf #2}\textcolor{black}{)}}}

\newcommand\RR{\mathbb{R}}

\newcommand\NN{\mathbb{N}}
\newcommand{\al}{\alpha}

\newcommand{\de}{\delta}
\newcommand{\ve}{\varepsilon}

\newcommand{\tht}{\theta}

\newcommand{\Om}{\Omega}

\newcommand{\La}{\Lambda}

\DeclareMathOperator{\spt}{spt}

\DeclareMathOperator*{\esssup}{ess\,sup}
\DeclareMathOperator*{\essinf}{ess\,inf}
\DeclareMathOperator*{\essosc}{ess\,osc}
\DeclareMathOperator{\tll}{Tail}
\DeclareMathOperator*{\tail}{\tll_\infty}

\newcommand{\abs}[1]{\left| #1\right|}

\newcommand{\lbr}[1][(]{\left#1}
\newcommand{\rbr}[1][)]{\right#1}

\usepackage[ddmmyyyy]{datetime}
\usepackage[margin=2cm]{geometry}
\makeatletter
\g@addto@macro\normalsize{%
  \setlength\abovedisplayskip{2pt}
  \setlength\belowdisplayskip{2pt}
  \setlength\abovedisplayshortskip{4pt}
  \setlength\belowdisplayshortskip{4pt}
}
\numberwithin{equation}{section}
\everymath{\displaystyle}
\usepackage[capitalize,nameinlink]{cleveref}
\crefname{section}{Section}{Sections}
\crefname{subsection}{Subsection}{Subsections}
\crefname{condition}{Condition}{Conditions}
\crefname{hypothesis}{Hypothesis}{Hypothesis}
\crefname{assumption}{Assumption}{Assumptions}
\crefname{lemma}{Lemma}{Lemmas}
\crefname{claim}{Claim}{Claims}
\crefname{remark}{Remark}{Remarks}

\crefformat{equation}{\textup{#2(#1)#3}}
\crefrangeformat{equation}{\textup{#3(#1)#4--#5(#2)#6}}
\crefmultiformat{equation}{\textup{#2(#1)#3}}{ and \textup{#2(#1)#3}}
{, \textup{#2(#1)#3}}{, and \textup{#2(#1)#3}}
\crefrangemultiformat{equation}{\textup{#3(#1)#4--#5(#2)#6}}%
{ and \textup{#3(#1)#4--#5(#2)#6}}{, \textup{#3(#1)#4--#5(#2)#6}}%
{, and \textup{#3(#1)#4--#5(#2)#6}}

\Crefformat{equation}{#2Equation~\textup{(#1)}#3}
\Crefrangeformat{equation}{Equations~\textup{#3(#1)#4--#5(#2)#6}}
\Crefmultiformat{equation}{Equations~\textup{#2(#1)#3}}{ and \textup{#2(#1)#3}}
{, \textup{#2(#1)#3}}{, and \textup{#2(#1)#3}}
\Crefrangemultiformat{equation}{Equations~\textup{#3(#1)#4--#5(#2)#6}}%
{ and \textup{#3(#1)#4--#5(#2)#6}}{, \textup{#3(#1)#4--#5(#2)#6}}%
{, and \textup{#3(#1)#4--#5(#2)#6}}

\crefdefaultlabelformat{#2\textup{#1}#3}

\newtheorem{theorem}{Theorem}[section]
\newtheorem{lemma}[theorem]{Lemma}
\newtheorem{claim}[theorem]{Claim}

\newtheorem{proposition}[theorem]{Proposition}

\newtheorem{definition}[theorem]{Definition}
\newtheorem{remark}[theorem]{Remark}        

\numberwithin{equation}{section}

\usepackage{enumitem}
\newlist{steps}{enumerate}{1}
\setlist[steps, 1]{label = \textcolor{Cerulean}{Step \arabic*:}}
\makeatletter
\def\ps@pprintTitle{%
	\let\@oddhead\@empty
	\let\@evenhead\@empty
	\def\@oddfoot{}%
	\let\@evenfoot\@oddfoot}
\makeatother
\DeclarePairedDelimiterX{\inp}[2]{\langle}{\rangle}{#1, #2}
\newcommand{\norm}[1]{\left\lVert#1\right\rVert}
\AtBeginDocument{%
\hypersetup{citecolor=red}}
\usepackage{doi}
\begin{document}

\begin{frontmatter}
\title{Local H\"older regularity for nonlocal parabolic $p$-Laplace equations}
\author[myaddress1]{Karthik Adimurthi\tnoteref{thanksfirstauthor}}
\ead{karthikaditi@gmail.com and kadimurthi@tifrbng.res.in}
\author[myaddress1]{Harsh Prasad\tnoteref{thankssecondauthor}}
\ead{harsh@tifrbng.res.in}
\author[myaddress2]{Vivek Tewary\tnoteref{thankssecondauthor}}
\ead{vivek.tewary@krea.edu.in}

\tnotetext[thanksfirstauthor]{Supported by the Department of Atomic Energy,  Government of India, under	project no.  12-R\&D-TFR-5.01-0520 and SERB grant SRG/2020/000081}

\tnotetext[thankssecondauthor]{Supported by the Department of Atomic Energy,  Government of India, under	project no.  12-R\&D-TFR-5.01-0520}

\address[myaddress1]{Tata Institute of Fundamental Research, Centre for Applicable Mathematics, Bangalore, Karnataka, 560065, India}
\address[myaddress2]{School of Interwoven Arts and Sciences, Krea University, Sri City, Andhra Pradesh, 517646, India}
\begin{abstract}
We prove local H\"older regularity for nonlocal parabolic equations of the form
\begin{align*}
	\partial_t u + \text{P.V.}\int_{\RR^N} \frac{|u(x,t)-u(y,t)|^{p-2}(u(x,t)-u(y,t))}{|x-y|^{N+ps}}\,dy=0,
\end{align*} for $p\in (1,\infty)$ and $s \in (0,1)$. Our contribution extends the fundamental work of Caffarelli, Chan and Vasseur on linear nonlocal parabolic equations to  equations modeled on the fractional $p$-Laplacian.
 
\end{abstract}

\end{frontmatter}
\tableofcontents

\section{Introduction}\label{sec1}

In this article, we prove local H\"older regularity for a nonlocal parabolic equation whose prototype equation is the parabolic fractional $p$-Laplacian of the form
\begin{align*}
    \partial_t u + \text{P.V.}\int_{\RR^N} \frac{|u(x,t)-u(y,t)|^{p-2}(u(x,t)-u(y,t))}{|x-y|^{N+ps}}\,dy=0,
\end{align*} with $p\in (1,\infty)$ and $s \in (0,1)$. The main tools in our investigation are
\begin{itemize} 
\item the exponential change of variable in time discovered by E.DiBenedetto, U.Gianazza and V.Vespri, which they used to prove Harnack's inequality for the parabolic $p$-Laplace equations. The standard references for these works are \cite{dibenedettoHarnackEstimatesQuasilinear2008,dibenedettoHarnackInequalityDegenerate2012}. The exponential change of variable in time leads to an expansion of positivity result which is suitable for proving H\"older regularity results. N.Liao in \cite{liaoUnifiedApproachHolder2020} has effected a direct proof of H\"older regularity for parabolic $p$-Laplace equation by using the expansion of positivity. It has also been used to give the first proofs of H\"older regularity for sign-changing solutions of doubly nonlinear equations of porous medium type in \cite{bogeleinHolderRegularitySigned2021,bogeleinHolderRegularitySigned2022,liaoHolderRegularitySigned2021}.  A recent work that employs the exponential change of variable to deduce regularity for elliptic anisotropic equations is \cite{liaoLocalRegularityAnisotropic2020}. 
\item the energy estimate for fractional parabolic $p$-Laplace equations containing the ``good term'' which substitutes the De Giorgi isoperimetric inequality in the so-called ``shrinking lemma''. The ``good term'' has been used to great effect in the works \cite{caffarelliDriftDiffusionEquations2010,caffarelliRegularityTheoryParabolic2011}. It has been used to define a novel De Giorgi class by Cozzi in \cite{cozziRegularityResultsHarnack2017} to prove H\"older regularity for fractional $p$-minimizers. Earlier, the absence of De Giorgi isoperimetric inequality was tackled by Di Castro et. al. through a logarithmic estimate \cite{dicastroLocalBehaviorFractional2016}. See \cite{cozziFractionalGiorgiClasses2019,adimurthiOlderRegularityFractional2022} for discussions on fractional versions of De Giorgi isoperimetric inequalities.
\end{itemize}

We demonstrate two purely nonlocal phenomenon - one at a technical level which has implications for what the corresponding De Giorgi conjectures related to the "best" Holder exponent in the nonlocal case should be (see \cref{rem:nonlocalA})  and another at the level of regularity  which says that in certain cases, time regularity can beat space regularity when we are working with nonlocal equations (see \cref{rem:nonlocalB}).

Previously, local boundedness for fractional parabolic $p$-Laplace equations is proved in the papers \cite{stromqvistLocalBoundednessSolutions2019,dingLocalBoundednessHolder2021,prasadLocalBoundednessVariational2021}. Regarding H\"older regularity for fractional parabolic $p$-Laplace equations in the spirit of De Giorgi-Nash-Moser, the only result that we are aware of is for the case $2 \leq p < \infty$ which was studied in \cite{dingLocalBoundednessHolder2021}, though we were unable to verify some of the calculations, particularly (5.23), pertaining to the logarithmic estimates in their paper. Explicit exponents for H\"older regularity for fractional parabolic $p$-Laplace equations with no coefficients appears in \cite{brascoContinuitySolutionsNonlinear2021}.

\subsection{A brief history of the problem} 

Much of the early work on regularity of fractional elliptic equations in the case $p=2$ was carried out by Silvestre \cite{silvestreHolderEstimatesSolutions2006}, Caffarelli and Vasseur \cite{caffarelliDriftDiffusionEquations2010}, Caffarelli, Chan, Vasseur \cite{caffarelliRegularityTheoryParabolic2011} and also Bass-Kassmann \cite{bassHarnackInequalitiesNonlocal2005,bassHolderContinuityHarmonic2005, kassmannPrioriEstimatesIntegrodifferential2009}. An early formulation of the fractional $p$-Laplace operator was done by Ishii and Nakamura \cite{ishiiClassIntegralEquations2010} and existence of viscosity solutions was established. Di Castro, Kuusi and Palatucci extended the De Giorgi-Nash-Moser framework to study the regularity of the fractional $p$-Laplace equation in \cite{dicastroLocalBehaviorFractional2016, dicastroNonlocalHarnackInequalities2014}. The subsequent work of Cozzi \cite{cozziRegularityResultsHarnack2017} covered a stable (in the limit $s\to 1$) proof of H\"older regularity by defining a novel fractional De Giorgi class.   Explicit exponents for H\"older regularity were found in \cite{brascoHigherHolderRegularity2018,brascoContinuitySolutionsNonlinear2021}. Other works of interest are \cite{iannizzottoGlobalHolderRegularity2016,defilippisHolderRegularityNonlocal2019,chakerRegularityNonlocalProblems2021,garainHigherOlderRegularity2022, defilippisGradientRegularityMixed2022a} and in the parabolic context, some works of interest are \cite{banerjeeLocalPropertiesSubsolution2021a, banerjeeLowerSemicontinuityPointwise2022a, dingLocalBoundednessHolder2021,prasadExistenceVariationalSolutions2022,prasadLocalBoundednessVariational2021}.

\subsection{On historical development of intrinsic scaling}

The method of intrinsic scaling was developed by E.DiBenedetto in \cite{dibenedettoLocalBehaviourSolutions1986} to prove H\"older regularity for degenerate quasilinear parabolic equations modelled on the $p$-Laplace operator.  A technical requirement of the proof was a novel logarithmic estimate which aids in the expansion of positivity.  Subsequently, the proof of H\"older regularity for the singular case was given in \cite{ya-zheLocalBehaviorSolutions1988} by switching the scaling from time variable to the space variable. These results were collected in E.DiBenedetto's treatise \cite{dibenedettoDegenerateParabolicEquations1993}. After a gap of several years, E.DiBenedetto, U.Gianazza and V.Vespri \cite{dibenedettoHarnackEstimatesQuasilinear2008, dibenedettoHarnackInequalityDegenerate2012} were able to prove Harnack's inequality for the parabolic $p$-Laplace equations by a new technique involving an exponential change of variable in time. This proof relies on expansion of positivity estimates and does not involve logarithmic test functions. Then, a new proof of H\"older regularity was given with a more geometric flavour in \cite{gianazzaNewProofHolder2010d}. This theory was extended to generalized parabolic $p$-Laplace equations with Orlicz growth in \cite{hwangHolderContinuityBounded2015,hwangHolderContinuityBounded2015a}.

\begin{remark}
    In this paper, we assume that solutions to the nonlocal equation are locally bounded. Local boundedness for the nonlocal elliptic equations is proved in \cite[Theorem 1.1]{dicastroLocalBehaviorFractional2016} and \cite[Theorem 6.2]{cozziRegularityResultsHarnack2017}. In the parabolic case, local boundedness is proved in \cite{stromqvistLocalBoundednessSolutions2019,dingLocalBoundednessHolder2021,prasadLocalBoundednessVariational2021}.
\end{remark}

\section{Notations and Auxiliary Results}\label{sec2}
In this section, we will fix the notation, provide definitions and state some standard auxiliary results that will be used in subsequent sections.

\subsection{Notations}
We begin by collecting the standard notation that will be used throughout the paper:
\begin{description}
\descitem{N1}{N1} We shall denote $N$ to be the space dimension and by $z=(x,t)$ to be  a point in $ \RR^N\times (0,T)$.  
\descitem{N2}{N2} We shall alternately use $\dfrac{\partial f}{\partial t},\partial_t f,f'$ to denote the time derivative of f.
\descitem{N3}{N3} Let $\Omega$ be an open bounded domain in $\mathbb{R}^N$ with boundary $\partial \Omega$ and for $0<T\leq\infty$,  let $\Omega_T\coloneqq \Omega\times (0,T)$. 
\descitem{N4}{N4} We shall use the notation
\begin{equation*}
	\begin{array}{ll}
	B_{\rho}(x_0)=\{x\in\RR^N:|x-x_0|<\rho\}, &
	\overline{B}_{\rho}(x_0)=\{x\in\RR^N:|x-x_0|\leq\rho\},\\
    I_{\tht}(t_0)=\{t\in\RR:t_0-\tht<t<t_0\},
    &Q_{\rho,\tht}(z_0)=B_{\rho}(x_0)\times I_\tht(t_0).
    \end{array}
\end{equation*} 
\descitem{N5}{N5} The maximum of two numbers $a$ and $b$ will be denoted by $a\wedge b\coloneqq \max(a,b)$. 
\descitem{N6}{N6} Integration with respect to either space or time only will be denoted by a single integral $\int$ whereas integration on $\Om\times\Om$ or $\RR^N\times\RR^N$ will be denoted by a double integral $\iint$. 
\descitem{N7}{N7} We will  use $\iiint$ to denote integral over $\RR^N \times \RR^N \times (0,T)$. More specifically, we will use the notation $\iiint_{Q}$ and $\iiint_{B \times I}$ to denote the integral over $\iiint_{B \times B \times I}$.
\descitem{N8}{N8} The notation $a \lesssim b$ is shorthand for $a\leq C b$ where $C$ is a universal constant which only depends on the dimension $N$, exponents $p$, and the numbers $\Lambda$ and $s$. 
\descitem{N9}{N9} For a function $u$ defined on the cylinder $Q_{\rho,\theta}(z_0)$ and any level $k \in \bb{R}$ we write $w_{\pm} = (u-k)_{\pm}$ 
\descitem{N10}{N10} For any fixed $t,k\in\RR$ and set $\Om\subset\RR^N$, we denote $A_{\pm}(k,t) = \{x\in \Om: (u-k)_{\pm}(\cdot,t) > 0\}$; for any ball $B_r$ we write $A_{\pm}(k,t) \cap B_{r} = A_{\pm}(k,r,t)$. 
\descitem{N11}{N11} For any set $\Om \subset \RR^N$, we denote $C_\Omega:=(\Omega^c\times\Omega^c)^c=\left(\Omega\times\Omega\right) \cup \left( \Omega\times(\RR^N\setminus\Omega)\right) \cup \left((\RR^N\setminus\Omega)\times\Omega\right)$.
\end{description}
\subsection{Structure of the equation}
Let $s \in (0,1)$ and $p >1$ be fixed, and   $\La \geq 1$ be a given constant. For almost every $x,y \in \RR^N$, let $K:\RR^N\times\RR^N\times \RR \to [0,\infty)$ be a symmetric measurable function satisfying
\begin{equation}\label{boundsonKernel}
    \frac{(1-s)}{\Lambda|x-y|^{N+ps}}\leq K(x,y,t)\leq \frac{(1-s)\Lambda}{|x-y|^{N+ps}}. 
\end{equation}

In this paper, we are interested in the regularity theory for the equation 
\begin{align}\label{maineq}
   \partial_t u + \mathcal{L}u=0,
\end{align} where $\mathcal{L}$ is the operator formally defined by 
\begin{equation*}
   \mathcal{L}u=\text{P.V.}\int_{\RR^N} K(x,y,t)|u(x,t)-u(y,t)|^{p-2}(u(x,t)-u(y,t))\,dy,\,x\in\RR^N.
\end{equation*}

\subsection{Function spaces}
Let $1<p<\infty$, we denote by $p'=p/(p-1)$ the conjugate exponent of $p$. Let $\Om$ be an open subset of $\RR^N$, we  define the {\it Sobolev-Slobodeki\u i} space, which is the fractional analogue of Sobolev spaces as follows:
\begin{align*}
    W^{s,p}(\Om):=\left\{ \psi\in L^p(\Omega): [\psi]_{W^{s,p}(\Om)}<\infty \right\},\qquad \text{for} \  s\in (0,1),
\end{align*} where the seminorm $[\cdot]_{W^{s,p}(\Om)}$ is defined by 
\begin{align*}
    [\psi]_{W^{s,p}(\Om)}:=\left( \iint_{\Om\times\Om} \frac{|\psi(x)-\psi(y)|^p}{|x-y|^{N+ps}}\,dx\,dy \right)^{\frac 1p}.
\end{align*}
The space when endowed with the norm $\norm{\psi}_{W^{s,p}(\Om)}=\norm{\psi}_{L^p(\Om)}+[\psi]_{W^{s,p}(\Om)}$ becomes a Banach space. The space $W^{s,p}_0(\Om)$ is the completion of $C_c^\infty(\Om)$ in the norm of $W^{s,p}(\RR^N)$. We will use the notation $W^{s,p}_{u_0}(\Om)$ to denote the space of functions in $W^{s,p}(\RR^N)$ such that $u-u_0\in W^{s,p}_0(\Om)$.

Let $I$ be an interval and let $V$ be a separable, reflexive Banach space, endowed with a norm $\norm{\cdot}_V$. We denote by $V^*$ to be its topological dual space. Let $v$ be a mapping such that for a.e. $t\in I$, $v(t)\in V$. If the function $t\mapsto \norm{v(t)}_V$ is measurable on $I$, then $v$ is said to belong to the Banach space $L^p(I;V)$ provided $\int_I\norm{v(t)}_V^p\,dt<\infty$. It is well known that the dual space $L^p(I;V)^*$ can be characterized as $L^{p'}(I;V^*)$.

Since the boundedness result requires some finiteness condition on the nonlocal tails, we define the tail space for some $m >0$ and $s >0$ as follows:
\begin{equation*}
    L^m_{s}(\RR^N):=\left\{ v\in L^m_{\text{loc}}(\RR^N):\int_{\RR^N}\frac{|v(x)|^m}{1+|x|^{N+s}}\,dx<+\infty \right\}.
\end{equation*}
Then a nonlocal tail is defined by 
\begin{equation*}
    \text{Tail}_{m,s,\infty}(v;x_0,R,I):=\text{Tail}_\infty(v;x_0,R,t_0-\tht,t_0):=\sup_{t\in (t_0-\tht, t_0)}\left( R^{sm}\int_{\RR^N\setminus B_R(x_0)} \frac{|v(x,t)|^{m-1}}{|x-x_0|^{N+sm}}\,dx \right)^{\frac{1}{m-1}},
\end{equation*} where $(x_0,t_0)\in \RR^N\times (-T,T)$ and the interval $I=(t_0-\tht,t_0)\subseteq (-T,T)$. From this definition, it follows that for any $v\in L^\infty(-T,T;L^{m-1}_{sm}(\RR^N))$, there holds $\text{Tail}_{m,s,\infty}(v;x_0,R,I)<\infty$. 

\subsection{Definition of weak solution}

Now, we are ready to state the definition of a weak sub(super)-solution.

\begin{definition}
    A function $u\in L^p_{\text{loc}}(I;W^{s,p}_{\text{loc}}(\Om))\cap C_{\text{loc}}(I;L^2_{\text{loc}}(\Om))\cap L^\infty_{\text{loc}}(I;L^{p-1}_{sp}(\RR^N))$ is said to be a local weak sub(super)-solution to \cref{maineq} in $\Omega\times I$ if for any closed interval $[t_1,t_2]\subset I$ and any compact set $\Omega' \subset \Omega$, the following holds:
    \begin{align*}
        \int_{\Om'}& u(x,t_2)\phi(x,t_2)\,dx - \int_{\Om'} u(x,t_1)\phi(x,t_1)\,dx - \int_{t_1}^{t_2}\int_{\Om'} u(x,t)\partial_t\phi(x,t)\,dx\,dt\\
        &+\int_{t_1}^{t_2}\iint_{\bb{R}^N\times\bb{R}^N}\,K(x,y,t)|u(x,t)-u(y,t)|^{p-2}(u(x,t)-u(y,t))(\phi(x,t)-\phi(y,t))\,dy\,dx\,dt
        \leq (\geq)0 ,
    \end{align*} for all $\phi\in L^p_{\text{loc}}(I,W^{s,p}_0(\Omega'))\cap W^{1,2}_{\text{loc}}(I,L^2(\Om'))$. 
\end{definition}

\subsection{Auxiliary Results}
We collect the following standard results which will be used in the course of the paper. We begin with the Sobolev-type inequality~\cite[Lemma 2.3]{dingLocalBoundednessHolder2021}.

\begin{theorem}\label{fracpoin}
    Let $t_2>t_1>0$ and suppose $s\in(0,1)$ and $1\leq p<\infty$. Then for any $f\in L^p(t_1,t_2;W^{s,p}(B_r))\cap L^\infty(t_1,t_2;L^2(B_r))$, we have
    \begin{align*}
        \int_{t_1}^{t_2}\fint_{B_r}|f(x,t)|^{p\left(1+\frac{2s}{N}\right)}\,dx\,dt
         \leq  C(N,s,p) &\left(r^{sp}\int_{t_1}^{t_2}\int_{B_r}\fint_{B_r}\frac{|f(x,t)-f(y,t)|^p}{|x-y|^{N+sp}}\,dx\,dy\,dt+\int_{t_1}^{t_2}\fint_{B_r}|f(x,t)|^p\,dx\,dt\right) \\
        & \times\left(\sup_{t_1<t<t_2}\fint_{B_r}|f(x,t)|^2\,dx\right)^{\frac{sp}{N}}.
    \end{align*}  
\end{theorem}

We also list a number of algebraic inequalities that are customary in obtaining energy estimates for nonlinear nonlocal equations.

\begin{lemma}(\cite[Lemma 4.1]{cozziRegularityResultsHarnack2017})\label{pineq1}
    Let $p\geq 1$ and $a,b\geq 0$, then for any $\tht\in [0,1]$,  the following holds:
    \begin{align*}
        (a+b)^p-a^p\geq \tht p a^{p-1}b + (1-\tht)b^p.
    \end{align*}
\end{lemma}

\begin{lemma}(\cite[Lemma 4.3]{cozziRegularityResultsHarnack2017})\label{pineq3}
    Let $p\geq 1$ and $a\geq b\geq 0$, then for any $\ve>0$, the following holds:
    \begin{align*}
        a^p-b^p\leq \ve a^p+\left(\frac{p-1}{\ve}\right)^{p-1}(a-b)^p.
    \end{align*}
\end{lemma}

Finally, we recall the following well known lemma concerning the geometric convergence of sequence of numbers (see \cite[Lemma 4.1 from Section I]{dibenedettoDegenerateParabolicEquations1993} for the details): 
\begin{lemma}\label{geo_con}
	Let $\{Y_n\}$, $n=0,1,2,\ldots$, be a sequence of positive number, satisfying the recursive inequalities 
	\[ Y_{n+1} \leq C b^n Y_{n}^{1+\alpha},\]
	where $C > 1$, $b>1$, and $\alpha > 0$ are given numbers. If 
	\[ Y_0 \leq  C^{-\frac{1}{\alpha}}b^{-\frac{1}{\alpha^2}},\]
	then $\{Y_n\}$ converges to zero as $n\to \infty$. 
\end{lemma}

\subsection{Main results}

We prove the following main theorems.

\begin{theorem}\label{holderparabolic}
    Let $p\in(2,\infty)$ and let $u$ be a locally bounded, local weak solution to \cref{maineq} in $\Om_T$. Then $u$ is locally H\"older continuous in $\Om_T$, i.e., there exist constants $C_0>1$, and $\beta\in (0,1)$ depending only on the data, such that with $Q_0:=B_{C_0R}\times (-(C_0R)^{2s},0)$ and $L$, $R_i$, $d_i$ for $i \in \{0,1,\ldots\}$ as defined in  \cref{defofL} and $\mathbf{C}_4$ as defined in \cref{eq:C4}, we have 
    \begin{align*}
        |u(x_1,t_1)-u(x_2,t_2)|\leq \frac{L}{R^{\beta}}  \left(\max\{|x_1-x_2|,L^{\frac{p-2}{sp}}\mathbf{C}_4|t_1-t_2|^{\frac{1}{sp}}\}\right)^{\beta}\,
    \end{align*} for every pair of points $(x_1,t_1), (x_2,t_2)\in B_R(0)\times (-d_1R^{sp},0)$.
\end{theorem}

\begin{theorem}\label{holderparabolicsin}
	Let $p\in(1,2)$ and let $u$ be a locally bounded, local weak solution to \cref{maineq} in $\Om_T$. Then $u$ is locally H\"older continuous in $\Om_T$, i.e., there exist constants $C_0>1$, and $\beta\in (0,1)$ depending only on the data, such that for a fixed $\ve>0$ with $Q_0:=B_{(C_0R)^{1-\ve_o}}\times (-(C_0R)^{sp},0)$ and $L$, $R_i$, $d_i$ for $i \in \{0,1,\ldots\}$ as defined in  \cref{defofLsin} and $\mathbf{C}_4$ as defined in \cref{eq:C4sin}, we have
	\begin{align*}
		|u(x_1,t_1)-u(x_2,t_2)|\leq \frac{L}{R^{\beta}}  \left(\max\{L^{\frac{2-p}{sp}}\mathbf{C}_4|x_1-x_2|,|t_1-t_2|^{\frac{1}{sp}}\}\right)^{\beta}\,
	\end{align*} for every pair of points $(x_1,t_1), (x_2,t_2)\in B_{d_1R}(0)\times (-R^{sp},0)$.
\end{theorem}

\section{Preliminary Results}\label{sec3}

\subsection{Energy estimates}

\begin{theorem}\label{energyest}
	Let $u$ be a local, weak sub(super)-solution in $B_{R}(x_0)\times (t_0-\tht,t_0)\Subset E_T$. Let $\tau_1,\tau_2 \in (0,\infty)$ be given such that $t_0-\tht<t_0-\tau_2<t_0-\tau_1\leq t_0$ and $0 < r < R$. Then there is positive universal constant $C(N,p,s,\La)$ such that for every level $k \in \bb{R}$ and every piecewise smooth cutoff function $\zeta(x,t)=\zeta_1(x)\zeta_2(t)$ with $\zeta_1 \in C_c^{\infty}(B_R)$, $\zeta_2 \in C^{0,1}(t_0-\theta,t_0)$, $0 \leq \zeta(x,t) \leq 1$, $\zeta_1 \equiv 1$ on $B_r(x_0)$,   we have
\begin{equation*}
	\def\arraystretch{1.6}
	\begin{array}{l}
    \underset{t_0-\tau_2<t <t_0-\tau_1 }{\esssup}\int_{B_r(x_0)}w_{\pm}^2\zeta^p(x,t)\,dx+ \int_{t_0-\tau_2}^{t_0-\tau_1}\int_{B_R(x_0)} \zeta^p(x,t)w_{\pm}(x,t)\int_{{B_{R}(x_0)}}\frac{w_{\mp}^{p-1}(y,t)}{|x-y|^{N+sp}}\,dy\,dx\,dt \\
    \hspace*{2cm}+\int_{t_0-\tau_2}^{t_0-\tau_1}\iint_{B_{r}(x_0)\times B_{r}(x_0)} \frac{|w_{\pm}(x,t)-w_{\pm}(y,t)|^p}{|x-y|^{N+ps}}\max\{\zeta_1(x),\zeta_1(y)\}^p\zeta_2^p(t)\,dx\,dy\,dt\\ 
        \hspace*{3cm}\leq  \int_{{B_R(x_0)}}w_{\pm}^2\zeta^p(x,t_0-\tau_2)\,dx\\
         \hspace*{4cm}+ C \int_{t_0-\tau_2}^{t_0-\tau_1}\iint_{{B_R(x_0)}\times {B_R(x_0)}} \max\{w_{\pm}(x,t),w_{\pm}(y,t)\}^p\frac{|\zeta_1(x)-\zeta_1(y)|^p}{|x-y|^{N+sp}}\zeta_2^p(t)\,dx\,dy\,dt \\ 
         \hspace*{4cm}+ C \int_{t_0-\tau_2}^{t_0-\tau_1}\int_{ {B_R(x_0)}} w_{\pm}^2(x,t)|\de_t \zeta(x,t)| \,dx\,dt \\
       \hspace*{4cm} + C\int_{t_0-\tau_2}^{t_0-\tau_1}\int_{ {B_R(x_0)}}\zeta^p(x,t)w_{\pm}(x,t) \lbr \esssup\limits_{t\in (t_0-\tau_2,t_0-\tau_1)}\int_{{\RR^N\setminus B_{R}(x_0)}} \frac{w_{\pm}(y,t)^{p-1}}{|x-y|^{N+sp}}\,dy \rbr\,dx\,dt,
    \end{array}
\end{equation*}
where $w_{\pm} = (u-k)_{\pm}$.
\end{theorem}
\begin{proof}
All the heavy-lifting pertaining to time regularization has already been performed in \cite[Lemma 3.3]{dingLocalBoundednessHolder2021}. In fact, the proof of our theorem is the same as their proof except for their estimate for the terms $I_2$ and $I_3$. For this reason and in the interest of brevity, we only present below the part of their proof with the subheadings ``{\bf The estimate of $I_2$}'' and ``{\bf The estimate of $I_3$}'' with different estimates that produce the ``good term''. To this end, we follow \cite[Prop. 8.5]{cozziRegularityResultsHarnack2017}. Moreover, we only write the calculation for sub-solutions since those for super-solutions are similar and we assume without loss of generality that $(x_0,t_0)=(0,0)$. Let $\rho_1, \rho_2$ be positive numbers such that $0<r\leq \rho_1<\rho_2\leq R$. We take $\zeta_1$ such that $0\leq \zeta_1\leq 1$,  $\spt(\zeta_1)=B_{\frac{\rho_1+\rho_2}{2}}:=B_{\frac{\rho_1+\rho_2}{2}}(0)$, $\zeta_1 \equiv 1$ on $B_{\rho_1}$ and $|\nabla \zeta_1| \leq \frac{1}{\rho_2-\rho_1}$. Moreover, let  $\tau_1,\tau_2 \in (0,\infty)$ be given such that $-\tht\leq-\tau_2<-\tau_1\leq 0$ and $d\mu:= \frac{dx\,dy}{|x-y|^{N+sp}}$.

\begin{description}
	\item[The estimate of $I_2$:] 

Recall that
\begin{align*}
    I_2:=\frac{1}{2}\int_{-\tau_2}^{-\tau_1} \iint_{B_{\rho_2}\times B_{\rho_2}}|u(x,t)-u(y,t)|^{p-2}(u(x,t)-u(y,t))(w_+(x,t)\zeta_1^p(x)-w_+(y,t)\zeta_1^p(y))\zeta_2^p(t)\,d\mu\,dt.
\end{align*}

To estimate the nonlocal terms, we consider the following cases pointwise for $(x,t)$ and $(y,t)$ in ${B_{\rho_2}} \times {B_{\rho_2}}\times I$ where $I=(-\tau_2,-\tau_1)$. Recalling the notation from \descref{N10}{N10}, for any fixed $t\in I$, we claim that
    \begin{itemize}
        \item If $x \notin A_+(k,\rho_2,t)$ and $y \notin A_+(k,\rho_2,t)$ then
        \begin{equation}\label{eq:A}
            |u(x,t)-u(y,t)|^{p-2}(u(x,t)-u(y,t))(w_+(x,t)\zeta_1^p(x)-w_+(y,t)\zeta_1^p(y))=0.
        \end{equation}
    This estimate follows from the fact that $(w_+(x,t)\zeta_1^p(x)-w_+(y,t)\zeta_1^p(y))=0$ when $x \notin A_+(k,t)$ and $y \notin A_+(k,t)$. 
        \item If $x \in A_+(k,\rho_2,t)$ and $ y \notin A_+(k,\rho_2,t)$ then
        \begin{multline}\label{eq:B}
            |u(x,t)-u(y,t)|^{p-2}(u(x,t)-u(y,t))(w_+(x,t)\zeta_1^p(x)-w_+(y,t)\zeta_1^p(y))\\
            \geq\min\{2^{p-2},1\}\left[|w_+(x,t)-w_+(y,t)|^p+w_-(y,t)^{p-1}w_+(x,t)\right]\zeta_1(x)^p.
        \end{multline}
    
    To obtain estimate \cref{eq:B}, we note that 
    \begin{multline*}
    	|u(x,t)-u(y,t)|^{p-2}(u(x,t)-u(y,t))(w_+(x,t)\zeta_1^p(x)-w_+(y,t)\zeta_1^p(y))\\
    	=(w_+(x,t)+w_-(y,t))^{p-1}w_+(x,t)\zeta_1^p(x),
    \end{multline*} when $x \in A_+(k,\rho_2,t)$ and $y \notin A_+(k,\rho_2,t)$. 
The estimate follows by an application of  Jensen's inequality when $p \leq 2$ and applying \cref{pineq1} with $\theta=0$ in the case $p \geq 2$ .
        \item If $x, y \in A_+(k,\rho_2,t)$ then
        \begin{align}\label{eq:C}
            |u(x,t)-u(y,t)|^{p-2}&(u(x,t)-u(y,t))(w_+(x,t)\zeta_1^p(x)-w_+(y,t)\zeta_1^p(y))\nonumber\\
            & \geq \frac{1}{2}|w_+(x,t)-w_+(y,t)|^p\max\{\zeta_1(x),\zeta_1(y)\}^p\nonumber\\
            & \qquad -C(p)\max\{w_+(x,t),w_+(y,t)\}^p|\zeta_1(x)-\zeta(y)|^p.
        \end{align}
    
    Since $x,y \in A_+(k,\rho_2,t)$, we have 
    \begin{multline*}
    	|u(x,t)-u(y,t)|^{p-2}(u(x,t)-u(y,t))(w_+(x,t)\zeta_1^p(x)-w_+(y,t)\zeta_1^p(y))\\
    	=\left\{ \begin{array}{ll}
    		(w_+(x,t)-w_+(y,t))^{p-1}(\zeta_1^p(x)w_+(x,t)-\zeta_1^p(y)w_+(y,t)) & \text{when} \, u(x,t) \geq u(y,t),\\
    		(w_+(y,t)-w_+(x,t))^{p-1}(\zeta_1^p(y)w_+(y,t)-\zeta_1^p(x)w_+(x,t)) & \text{when} \, u(x,t) \leq u(y,t).
    	\end{array}\right.
    \end{multline*}
    Thus, without loss of generality, we can assume {$u(x,t) \geq u(y,t)$}.  Thus, we  get
    \begin{multline*}
    	|u(x,t)-u(y,t)|^{p-2}(u(x,t)-u(y,t))(w_+(x,t)\zeta_1^p(x)-w_+(y,t)\zeta_1^p(y))\\
    	=\left\{\hspace*{-6pt}\begin{array}{ll}
    		(w_+(x,t)-w_+(y,t))^{p}\zeta_1^p(x) & \text{if} \, \zeta_1(x) \geq \zeta_1(y), \\
    		(w_+(x,t)-w_+(y,t))^{p}(\zeta_1(y))^p-(w_+(x,t)-w_+(y,t))^{p-1}w_+(x,t)(\zeta_1^p(x)-\zeta_1^p(y)) & \text{if} \, \zeta_1(x) \leq \zeta_1(y).
    	\end{array}\right.
    \end{multline*} 
    In the second case, we apply \cref{pineq3} with $a=\zeta_1(y)$ and $b=\zeta_1(x)$ and 
    $\ve=\tfrac{1}{2}\tfrac{w_+(x,t)-w_+(y,t)}{w_+(x,t)}$  to obtain
    \begin{multline*}
    	(w_+(x,t)-w_+(y,t))^{p-1}w_+(x,t)(\zeta_1^p(x)-\zeta_1^p(y))\\
    	\leq \frac{1}{2}(w_+(x,t)-w_+(y,t))^p\zeta_1^p(y)+[2(p-1)]^{p-1}w_+^p(x,t)(\zeta_1(y)-\zeta(x))^p.
    \end{multline*}
    Combining the previous three estimates,  we obtain \cref{eq:C} when {$u(x,t) \geq u(y,t)$}.  The case {$u(x,t) \leq u(y,t)$}  can be handled analogously by interchanging the role of $x$ and $y$
    \end{itemize}

    As a consequence of these cases and \cref{boundsonKernel}, we have 
    \begin{align*}
    I_2\geq& C_{11}\Biggr[\int_I\iint_{B_{\rho_2}\times B_{\rho_2}} \frac{|w_+(x,t)-w_+(y,t)|^p}{|x-y|^{N+ps}}\max\{\zeta_1(x),\zeta_1(y)\}^p\zeta_2^p(t)\,dx\,dy\,dt \nonumber\\
    &\qquad+\int_I\int_{B_{\rho_2}}\zeta_2^p(t)\zeta_1^p(x)w_+(x,t)\left(\int_{B_{\rho_2}}\frac{w_-(y,t)^{p-1}}{|x-y|^{N+sp}}\,dy\right)\,dx\,dt\Biggr]\\
    &\qquad\qquad-C_{22}\int_I\iint_{B_{\rho_2}\times B_{\rho_2}}\max\{w_+(x,t),w_+(y,t)\}^p\frac{|\zeta_1(x)-\zeta_1(y)|^p}{|x-y|^{N+sp}}\zeta_2^p(t)\,dx\,dy\,dt.
\end{align*}

\item[The estimate of $I_3$:] Recall that
\begin{align}\label{caccest3}
    I_3:&=\int_{-\tau_2}^{-\tau_1} \iint_{(\RR^N\setminus B_{\rho_2})\times B_{\rho_2}}|u(x,t)-u(y,t)|^{p-2}(u(x,t)-u(y,t))(w_+(x,t)\zeta_1^p(x)-w_+(y,t)\zeta_1^p(y))\zeta_2^p(t)\,d\mu\,dt\nonumber\\
    &=\int_{-\tau_2}^{-\tau_1}\int_{B_{\rho_2}}\zeta_2^p(t)\zeta_1^p(x)w_+(x,t)\left[\int_{\RR^N\setminus B_{\rho_2}}{|u(x,t)-u(y,t)|^{p-2}(u(x,t)-u(y,t))}K(x,y,t)\,dy\right]\,dx\,dt\nonumber\\
    &\geq C_{33}\underbrace{\int_{-\tau_2}^{-\tau_1} \int_{A^+(k,\rho_1,t)}\zeta_2^p(t)\zeta_1^p(x)w_+(x,t)\left[\int_{(B_R(x)\setminus B_{\rho_2}) \cap \{u(x,t)\geq u(y,t)\}} \frac{(u(x,t)-u(y,t))^{p-1}}{|x-y|^{N+sp}}\,dy\right]\,dx\,dt}_{T_1}\nonumber\\
    &\qquad - C_{44}\underbrace{\int_{-\tau_2}^{-\tau_1} \int_{A^+\left(k,\frac{\rho_1+\rho_2}{2},t\right)}\zeta_2^p(t)\zeta_1^p(x)w_+(x,t)\left[\int_{(\RR^N\setminus B_{\rho_2}) \cap \{u(y,t)> u(x,t)\}} \frac{(u(y,t)-u(x,t))^{p-1}}{|x-y|^{N+sp}}\,dy\right]\,dx\,dt}_{T_2}.
\end{align}
We estimate $T_1$ as follows:
\begin{align}\label{caccest4}
    T_1&\geq \int_{-\tau_2}^{-\tau_1}\int_{B_{\rho_1}}\zeta_2^p(t)\zeta_1^p(x)w_+(x,t)\left[\int_{(B_R(x)\cap A^-(k,t))\setminus B_{\rho_2}} \frac{(w_+(x,t)+w_-(y,t))^{p-1}}{|x-y|^{N+sp}}\,dy\right]\,dx\,dt\nonumber\\
    &\geq \int_{-\tau_2}^{-\tau_1}\int_{B_{\rho_1}}\zeta_2^p(t)\zeta_1^p(x)w_+(x,t)\left[\int_{B_R(x)\setminus B_{\rho_2}} \frac{w_-(y,t)^{p-1}}{|x-y|^{N+sp}}\,dy\right]\,dx\,dt.
\end{align}
We estimate $T_2$ as follows:
\begin{align}\label{caccest5}
    T_2&\leq \int_{-\tau_2}^0 \int_{B_{\frac{\rho_1+\rho_2}{2}}}\zeta_2^p(t)\zeta_1^p(x)w_+(x,t)\left[\int_{{\RR^N\setminus B_{\rho_2}}} \frac{w_+(y,t)^{p-1}}{|x-y|^{N+sp}}\,dy\right]\,dx\,dt\nonumber\\
    &\leq \iint_{I\times B_R}\zeta_2^p(t)\zeta_1^p(x)w_+(x,t)\lbr \esssup\limits_{t\in (-\tau_2,0)}\left[\int_{{\RR^N\setminus B_{\rho_2}}} \frac{w_+(y,t)^{p-1}}{|x-y|^{N+sp}}\,dy\right]\rbr\,dx\,dt.
\end{align}

Combining this observation with \cref{caccest4} and \cref{caccest5} into \cref{caccest3}, we obtain
\begin{align*}
    I_3\geq C_{33}&\underbrace{\int_{I}\int_{B_{\rho_1}}\zeta_2^p(t)\zeta_1^p(x)w_+(x,t)\left[\int_{B_R(x)\setminus B_{\rho_2}} \frac{w_+(y,t)^{p-1}}{|x-y|^{N+sp}}\,dy\right]\,dx\,dt}_{\text{ignore this term which appears on the left hand side}}\\
    &-C_{44} \iint_{I\times B_R}\zeta_2^p(t)\zeta_1^p(x)w_+(x,t) \lbr \esssup\limits_{t\in (-\tau_2,0)}\int_{{\RR^N\setminus B_{\rho_2}}} \frac{w_+(y,t)^{p-1}}{|x-y|^{N+sp}}\,dy \rbr\,dx\,dt.
\end{align*}
\end{description}
As mentioned earlier, the rest of the proof of \cref{energyest} is the same as in \cite[Lemma 3.3]{dingLocalBoundednessHolder2021}.
\end{proof}

\begin{remark}
This energy estimate first appears in \cite{prasadLocalBoundednessVariational2021} under the assumption that $\partial_t u\in L^2(E_T)$ (see Lemma 3.1 and Remark 3.2 in \cite{prasadLocalBoundednessVariational2021}). This assumption is generally dropped by working with regularizations in time, such as Steklov averages \cite{dibenedettoDegenerateParabolicEquations1993}. In fact, the estimate in \cite{prasadLocalBoundednessVariational2021} is proved for parabolic minimizers, however the same method works for solutions of equations.
\end{remark}
Henceforth, excepting the final section, we will let $u$ denote a nonnegative, locally bounded supersolution.

\subsection{Shrinking Lemma}

One of the main difficulties we face when dealing with regularity issues for nonlocal equations is the lack of a corresponding isoperimetric inequality for $W^{s,p}$ functions. Indeed, since such functions can have jumps, a generic isoperimetric inequality seems out of reach at this time (see \cite{cozziFractionalGiorgiClasses2019, adimurthiOlderRegularityFractional2022}). One way around this issue is to note that since we are working with solutions of an equation, which we expect to be continuous and hence have no jumps, we could try and cook up an isoperimetric inequality for solutions. Such a strategy turns out to be feasible due to the presence of the ``good term'' or the isoperimetric term in the Caccioppoli inequality. 

\begin{lemma}\label{lem:isop}
Let $k<l<m$ be arbitrary levels and $A \geq 1$. Then,
\[
(l-k)(m-l)^{p-1}\abs{[u>m]\cap B_{\rho}}\abs{[u<k]\cap B_{\rho}} \leq 
C\rho^{N+sp}\int_{B_{\rho}} (u-l)_{-}(x)\int_{B_{A\rho}}\frac{(u-l)_{+}^{p-1}(y)}{|x-y|^{N+sp}}\,dy\,dx,
\]
where $C = C(N,s,p,A)>0$. 
\end{lemma}

\begin{proof}
We estimate from below
\begin{align*}
    \int_{B_{\rho}} (u-l)_{-}(x)\int_{B_{A\rho}}\frac{(u-l)_{+}^{p-1}(y)}{|x-y|^{N+sp}}\,dy\,dx &\geq \frac{C}{\rho^{N+sp}}\int_{B_{\rho}} (u-l)_{-}(x)\int_{B_{\rho}}(u-l)_{+}^{p-1}(y)\,dy\,dx
    \\
    &\geq \frac{C}{\rho^{N+sp}}\int_{B_{\rho}\cap \{u<k\}} (u-l)_{-}(x)\int_{B_{\rho}\cap \{u>m\}}(u-l)_{+}^{p-1}(y)\,dy\,dx
    \\
    &\geq \frac{C}{\rho^{N+sp}}\int_{B_{\rho}\cap \{u<k\}} (l-k)\int_{B_{\rho}\cap \{u>m\}}(m-l)^{p-1}\,dy\,dx.
\end{align*}
\end{proof}

We now prove the shrinking lemma using \cref{lem:isop}. It is usually proved in the local case using the De Giorgi isoperimetric inequality.
\begin{lemma}\label{lem:shrinking}
Let $u$ be a super-solution of \cref{maineq}. Suppose that for some level $m$, some constant $\nu \in (0,1)$ and all time levels $\tau$ in some interval $J$ we have
\[
|[u(\cdot,\tau)>m]\cap B_{\rho}| \geq \nu|B_{\rho}|,
\]
and we can arrange that for some $A \geq 1$,  the following is also satisfied:
\begin{align}\label{smallness}
\int_J\int_{B_{\rho}} (u-l)_{-}(x,t)\int_{B_{A\rho}(x)}\frac{(u-l)_{+}^{p-1}(y,t)}{|x-y|^{N+sp}}\,dy\,dx\,dt \leq C_1\frac{l^p}{\rho^{sp}}|B_{\rho} \times J|,
\end{align}
where $l = \frac{m}{2^{j}}$ for some $j\geq 1$, 
then we have the following conclusion:
\[
\left|\left[u<\frac{m}{2^{j+1}}\right]\cap B_{\rho} \times J\right| \leq \left(\frac{C_{(C_1,A,N,\nu)}}{2^{j}-1}\right)^{p-1}|B_{\rho} \times J|.
\]
\end{lemma}
\begin{proof}
In the conclusion of \cref{lem:isop} we put $k = \frac{m}{2^{j+1}}$ and $l = \frac{m}{2^{j}}$, use the hypothesis and then integrate over the time interval $J$ to get
\[
\frac{m}{2^{j+1}}m^{p-1}\left(\frac{2^j-1}{2^j}\right)^{p-1} \left|\left[u<\frac{m}{2^{j+1}}\right]\cap B_{\rho} \times J\right| \leq 
\frac{C}{\nu}\left(\frac{m}{2^j}\right)^{p}|B_{\rho} \times J|,
\]
where $C$ depends on $C_1,A$ and $N$. The conclusion follows after a simple rearrangement. 
\end{proof}

\begin{remark}
In applying the shrinking lemma in proving the expansion of positivity lemmas, in order to get the smallness condition \cref{smallness}, we shall need to impose the smallness condition on the Tail term - this is one of the ways in which the Tail alternatives enter the picture.
\end{remark}

\begin{remark}\label{rem:nonlocalA}
An intriguing aspect of our lemma is the fact that the dependence between the levels in the conclusion and how small we can make the corresponding level set is \textit{not} exponential; such a dependence is reflected in the best known lower bound for the H\"older exponent for solutions to second order linear nonlocal equations \cite{mosconiOptimalEllipticRegularity2018}, whereas, in the local case, this lower bound is much worse \cite{bombieriHarnackInequalityElliptic1972}.
\end{remark}


\subsection{Tail Estimates}\label{sec:tail}

In this section, we shall outline how the  estimate for the tail term is made and we shall refer to this section whenever a similar calculation is required in subsequent sections.

Given a function $ \zeta \in C_c^{\infty}(B_{\rho})$ and  any level $k$, we want to estimate
\[
\underset{\stackrel{t \in J}{x\in \spt\zeta}}{\esssup}\int_{B_{\rho}^c(y_0)}\frac{(u-k)_{-}^{p-1}(y,t)}{|x-y|^{N+sp}}\,dy ,
\]
where $J$ is some time interval.  In order to do this, we will typically choose $\zeta$ to be supported in $B_{\vartheta\rho}$ for some $\vartheta \in (0,1)$. With such a choice, we have
\[
|y-y_0| \leq |x-y|\left(1+\frac{|x-y_0|}{|x-y|}\right)\leq  |x-y|\left(1+\frac{\vartheta}{(1-\vartheta)}\right),
\]
so we can make a first estimate
\[
\underset{\stackrel{t \in J}{ x\in \spt\zeta}}{\esssup}\int_{B_{\rho}^c(y_0)}\frac{(u-k)_{-}^{p-1}(y,t)}{|x-y|^{N+sp}}\,dy 
\leq \frac{1}{(1-\vartheta)^{N+sp}}\underset{t \in J}{\esssup}\int_{B_{\rho}^c(y_0)}\frac{(u-k)_{-}^{p-1}(y,t)}{|y-y_0|^{N+sp}}\,dy. 
\]
In the local case, we could always estimate $(u-k)_{-} \leq k$ because we take $u \geq 0$ locally. However, in the Tail term we have no information regarding the solution outside the ball and so unless we make a global boundedness assumption (for e.g. $u \geq 0$ in full space), the best we can do is
\[
(u-k)_{-} \leq u_{-} + k,
\]
which leads us to the next estimate
\[
\underset{t \in J}{\esssup}\int_{B_{\rho}^c(y_0)}\frac{(u-k)_{-}^{p-1}(y,t)}{|y-y_0|^{N+sp}}\,dy  \leq C(p)\frac{k^{p-1}}{\rho^{sp}} +  C(p)\,\underset{t \in J}{\esssup}\int_{B_{\rho}^c(y_0)}\frac{u_{-}^{p-1}(y,t)}{|y-y_0|^{N+sp}}\,dy.
\]
Putting together the above estimates yield
\[
\underset{\stackrel{t \in J}{ x\in \spt\zeta}}{\esssup}\int_{B_{\rho}^c(y_0)}\frac{(u-k)_{-}^{p-1}(y,t)}{|x-y|^{N+sp}}\,dy  \leq \frac{C(p)}{\rho^{sp}}\frac{1}{(1-\vartheta)^{N+sp}}\left[k^{p-1}+\tll_{\infty}^{p-1}(u_{-};y_0,\rho,J)\right].
\]
Finally we usually want an estimate of the form
\[
\underset{\stackrel{t \in J}{ x\in \spt\zeta}}{\esssup}\int_{B_{\rho}^c(y_0)}\frac{(u-k)_{-}^{p-1}(y,t)}{|x-y|^{N+sp}}\,dy \leq \frac{C(p)}{(1-\vartheta)^{N+sp}}\frac{k^{p-1}}{\rho^{sp}},
\]
and so we impose the condition that
\[
\tll_{\infty}^{p-1}(u_{-};y_0,\rho,J) \leq k^{p-1}.
\]
This is the origin of the various Tail alternatives. 

\begin{remark}
We will see that it is the Tail that captures the nonlocality of our equation in the sense that if it is small then the proofs become local. In fact, when dealing with nonnegative supersolutions, we only need the Tail of the negative part of the solution to be small and so if we assume that it is nonnegative in full space then the Tail alternatives are automatically verified and the  proofs of the constituent lemmas become ``local proofs''.
\end{remark}
%
\subsection{De Giorgi iteration Lemma} 
Let 
\[
\mu_{-} \leq \underset{(x_0,t_0)+Q_{2\rho}^{-}(\theta)}{\essinf} u
\]
and $M > 0$. Let $\xi \in (0,1]$ and $a\in (0,1)$ be fixed numbers. 
\begin{lemma}\label{lem:deG}
Let $u$ be a locally bounded, local, weak supersolution of \cref{maineq} in $E_T$. Then there exists a number $\nu_{-}$ depending only on the data and the parameters $\theta,\xi,M$ and $a$ such that if the following is satisfied
\[
|[u\leq \mu_{-} + \xi M]\cap (x_0,t_0)+Q_{2\rho}^{-}(\theta)| \leq \nu_{-}|Q_{2\rho}^{-}(\theta)|,
\]
then one of the following two conclusion holds: 
\[
\tail((u-\mu)_{-};x_0,\rho,(t_0-\theta (2\rho)^{sp},t_0]) > \xi M,
\]
or
\[
u \geq \mu_{-} + a\xi M \text{  a.e. in  } [(x_0,t_0)+Q_{\rho}^{-}(\theta)].
\]
\end{lemma}
\begin{remark}
	An analogous result holds for subsolutions at the supremum. 
\end{remark}
\begin{proof}
Without loss of generality, we will assume that $(x_0,t_0) = (0,0)$ and for $n=0,1,2,\ldots$, we set
\begin{equation*}
    \rho_n = \rho + 2^{-n}\rho,  \qquad
    B_n  = B_{\rho_n} \quad \text{and} \quad
    Q_n = B_n\times(-\theta\rho_n^{sp},0].
\end{equation*}
We apply the energy estimates over $B_n$ and $Q_n$ to $ (u-b_n)_{-}$ for the levels
\begin{equation*}
    b_n = \mu_{-} + \xi_nM \qquad \text{ where } \qquad \xi_n =  a\xi + \frac{1-a}{2^n}\xi.
\end{equation*}
Let us define 
\[
\hat{\rho_n} := \frac{3\rho_n+\rho_{n+1}}{4}, \quad \tilde{\rho_n} := \frac{\rho_n+\rho_{n+1}}{2}, \quad \text{and} \quad \bar{\rho_n} := \frac{\rho_n+3\rho_{n+1}}{4},
\]
and denote $\hat{Q_n}:= B_{\hat{\rho_n}}\times (-\theta\hat\rho_n^{sp}, 0]$, $\tilde{Q_n}:= B_{\tilde{\rho_n}}\times (-\theta\tilde\rho_n^{sp}, 0]$ and $\bar{Q_n}:= B_{\bar{\rho_n}}\times (-\theta\bar\rho_n^{sp}, 0]$. 

We shall consider cut-off functions of the form  $\bar{\zeta_n}$ and $\zeta_n$ satisfying
\begin{equation}\label{cutoff_size}
	\begin{array}{llll}
		\bar{\zeta_n} \equiv 1 \text{ on } B_{n+1}, \quad \bar{\zeta_n} \in C_c^{\infty}(B_{\bar{\rho_n}}), \quad  |\nabla\bar{\zeta_n}| \apprle \frac{1}{\bar{\rho_n} - \rho_{n+1}} \approx \frac{2^n}{\rho} \quad \text{and} \quad  |\de_t\bar{\zeta_n}| \apprle \frac{1}{\theta(\bar\rho_n^{sp} - \rho_{n+1}^{sp})} \approx \frac{2^{nsp}}{\theta\rho^{sp}}, \\
				{\zeta_n} \equiv 1 \text{ on } B_{\tilde{\rho_n}}, \quad {\zeta_n} \in C_c^{\infty}(B_{\hat{\rho_n}}), \quad  |\nabla{\zeta_n}| \apprle \frac{1}{\hat{\rho_n} - \tilde{\rho_{n}}}\approx \frac{2^{n}}{\rho} \quad \text{and} \quad  |\de_t{\zeta_n}| \apprle \frac{1}{\theta(\hat\rho_n^{sp} - \tilde{\rho_{n}}^{sp})}\approx \frac{2^{nsp}}{\theta\rho^{sp}},
	\end{array}
\end{equation}
Let us also denote $A_n:= [u < b_n] \cap Q_n$, then we can apply \cref{fracpoin} to get
\begin{equation}\label{eq3.8}
	\begin{array}{rcl}
	\frac{(1-a)\xi M}{2^n} |A_{n+1}| & \leq & \iint_{Q_{n+1}} (u-b_{n})_- \ dz \leq  \iint_{\bar{Q_{n}}} (u-b_{n})_- \bar{\zeta_n} \ dz \\
	& \leq & \lbr \iint_{\tilde{Q_{n}}} \lbr (u-b_{n})_- \bar{\zeta_n}\rbr^{p\frac{N+2s}{N}} \ dz\rbr^{\frac{N}{p(N+2s)}} |A_n|^{\frac{p(N+2s)-N}{p(N+2s)}} \\
	& \apprle & \bigg[ \tilde\rho_n^{sp} \iiint_{\tilde{Q_n} \times B_{\tilde{\rho_n}}}\frac{|(u-b_{n})_- \bar{\zeta_n}(x,t) - (u-b_{n})_- \bar{\zeta_n}(y,t)|^p}{|x-y|^{N+sp}}\\
	&& +    \iint_{\tilde{Q_n}}((u-b_{n})_- \bar{\zeta_n})^p\bigg]^{\frac{N}{p(N+2s)}}\lbr \sup_{-\theta \tilde\rho_n^{sp}< t < 0}\fint_{B_{\tilde{\rho_n}}} ((u-b_{n})_- \bar{\zeta_n})^2 \ dz\rbr^{\frac{s}{N+2s}}|A_n|^{\frac{p(N+2s)-N}{p(N+2s)}}
		\end{array}
\end{equation}

From Young's inequality, we have
\begin{multline}\label{eq3.9}
	|(u-b_{n})_- \bar{\zeta_n}(x,t) - (u-b_{n})_- \bar{\zeta_n}(y,t)|^p \leq c |(u-b_{n})_- (x,t) - (u-b_{n})_- (y,t)|^p\bar{\zeta_n}^p(x,t) \\
	+ c |(u-b_{n})_-(y,t)|^p |\bar{\zeta_n}(x,t) - \bar{\zeta_n}(y,t)|^p
\end{multline}
Combining \cref{eq3.8} and \cref{eq3.9}, we get
\begin{equation}\label{eq3.10}
	\begin{array}{rcl}
	\frac{(1-a)\xi M}{2^n} |A_{n+1}| & \apprle & \bigg[ \tilde\rho_n^{sp} \iiint_{\tilde{Q_n} \times B_{\tilde{\rho_n}}}\hspace*{-0.5cm}\frac{|(u-b_{n})_- (x,t) - (u-b_{n})_- (y,t)|^p}{|x-y|^{N+sp}}\\
	&& + \tilde\rho_n^{sp} \iiint_{\tilde{Q_n} \times B_{\tilde{\rho_n}}}\hspace*{-0.5cm}\frac{|(u-b_{n})_-(y,t)|^p |\bar{\zeta_n}(x,t) - \bar{\zeta_n}(y,t)|^p}{|x-y|^{N+sp}}\\
	&&  +    \iint_{\tilde{Q_n}}((u-b_{n})_- )^p\bigg]^{\frac{N}{p(N+2s)}} \lbr \sup_{-\theta \tilde\rho_n^{sp}< t < 0}\fint_{B_{\tilde{\rho_n}}} (u-b_{n})_-^2 \ dz\rbr^{\frac{s}{N+2s}}|A_n|^{\frac{p(N+2s)-N}{p(N+2s)}}.
		\end{array}
\end{equation}


The energy inequality applied with the cutoff function $\zeta_n$ over $Q_n$ from \cref{energyest} gives
\begin{multline}\label{eq3.11}
	\underset{-\theta\tilde\rho_n^{sp} < t < 0}{\esssup}\int_{B_{\tilde{\rho_n}}}(u-b_n)_{-}^2\,dx  + {\iiint_{\tilde{Q_n} \times B_{\tilde{\rho_n}}}} \hspace*{-0.5cm}\frac{|(u-b_n)_{-}(x,t)(x,t)-(u-b_n)_{-}(y,t)|^p}{|x-y|^{N+sp}}\,dx \,dy\,dt 
	\\ 
	\leq \underbrace{\int_{-\theta\rho_n^{sp}}^0 \iint_{B_{\rho_n} \times B_{\rho_n}} \max\{(u-b_n)_-(x,t),(u-b_n)_-(y,t)\} \frac{|\zeta_n(x,t) - \zeta_n(y,t)|^p}{|x-y|^{N+sp}}\,dx\,dy\,dt}_{:=I}\\
	+ \underbrace{\int_{-\theta\rho_n^{sp}}^0 \int_{B_{\rho_n}} (u-b_n)_-^2 |\pa_t \zeta_n(x,t)| \,dx\,dt}_{:=II} \\
	+ \underbrace{\int_{-\theta\rho_n^{sp}}^0 \int_{B_{\rho_n}}(u-b_n)_-(x,t) \,dx\,dt \lbr \esssup_{\stackrel{-\theta\rho_n^{sp} < t < 0}{x \in \spt \zeta_n}} \int_{\RR^N \setminus B_{\rho_n}} \frac{(u-b_n)_-^{p-1}}{|x-y|^{N+sp}}\,dy\rbr}_{:=III}. 
\end{multline}
We now estimate each of the terms in the previous display as follows:
\begin{description}
	\item[Estimate for $I$:] Using the fact that $(u-b_n)_- \leq \xi M$ and \cref{cutoff_size}, we have
	\begin{equation}\label{eq3.13}
		I \leq C \frac{2^{np} (\xi M)^p}{\rho^{sp}} |A_n|.
	\end{equation}
		\item[Estimate for $II$:] Analogously, we also get
			\begin{equation}\label{eq3.14}
			II \leq C \frac{2^{nsp} (\xi M)^2}{\theta \rho^{sp}} |A_n|.
		\end{equation}
			\item[Estimate for $III$:] To estimate this, we note that if $x \in B_{\tilde{\rho_n}}$ and $y \in \RR^N \setminus B_n$, then we have 
			\[
			|y|\leq |x-y|\left(1+\frac{|x|}{|x-y|}\right) \leq c2^n|x-y| \ \Longrightarrow\  \frac{1}{|x-y|^{N+sp}} \leq c2^{n(N+sp)}\frac{1}{|y|^{N+sp}}.
			\]
			Next, we note that the following
			\[
			(u-b_n)_{-} \leq \xi_nM - (u - \mu_{-}) \leq \xi M + (u - \mu_{-})_{-},
			\]
			holds globally. Therefore we have
			\begin{equation}\label{eq315.15}\begin{array}{l}
				\underset{\stackrel{-\theta\rho_n^{sp} < t < 0;}{ x\in \spt\zeta_1}}{\esssup}\int_{\RR^N \setminus B_{\rho_n}}\frac{(u-b_n)_{-}^{p-1}(y,t)}{|x-y|^{N+sp}}\,dy\\
				 \hspace*{2cm}\apprle 2^{n(N+sp)}\underset{-\theta\rho_n^{sp} < t < 0}{\esssup}\int_{\RR^N \setminus B_{\rho_n}}\frac{(u - \mu_{-})_{-}^{p-1}(y,t)+(\xi M)^{p-1}}{|y|^{N+sp}}\,dy
				\\
				\hspace*{2cm}\apprle 2^{n(N+sp)}\frac{(\xi M)^{p-1}}{\rho^{sp}} +
				2^{n(N+sp)}\underset{-\theta\rho_n^{sp} < t < 0 }{\esssup}\int_{\RR^N \setminus B_{\rho_n}}\frac{(u - \mu_{-})_{-}^{p-1}(y,t)}{|y|^{N+sp}}\,dy
				\\
				\hspace*{2cm}\leq c2^{n(N+sp)}\left(\frac{(\xi M)^{p-1}}{\rho^{sp}} + \frac{1}{\rho^{sp}}\tll_{\infty}^{p-1}((u-\mu)_{-};0,\rho,(-\theta (2\rho)^{sp},0])\right),
				\end{array}
			\end{equation}
		In particular, we get
		\begin{equation}\label{eq3.15}
			\begin{array}{rcl}
			III &\apprle & 2^{n(N+sp)}\left(\frac{(\xi M)^{p-1}}{\rho^{sp}} + \frac{1}{\rho^{sp}}\tll_{\infty}^{p-1}((u-\mu)_{-};0,2\rho,(-\theta (2\rho)^{sp},0])\right)(\xi M)|A_n|
			\\			
			&\leq&  C 2^{n(N+sp)}\frac{(\xi M)^{p}}{\rho^{sp}}|A_n|, 
			\end{array}
		\end{equation}
	where to obtain the last estimate, we made use of the hypothesis $$\tail((u-\mu)_{-};x_0,\rho,(-\theta (2\rho)^{sp},0]) \leq  \xi M.$$
\end{description}

Combining \cref{eq3.13}, \cref{eq3.14} and \cref{eq3.15} into \cref{eq3.11}, we get
\begin{multline}\label{eq3.16}
	\underset{-\theta\tilde\rho_n^{sp} < t < 0}{\esssup}\int_{B_{\tilde{\rho_n}}}(u-b_n)_{-}^2\,dx  + {\iiint_{\tilde{Q_n} \times B_{\tilde{\rho_n}}}} \hspace*{-0.5cm}\frac{|(u-b_n)_{-}(x,t)(x,t)-(u-b_n)_{-}(y,t)|^p}{|x-y|^{N+sp}}\,dx \,dy\,dt 
	\\
	\leq  \underbrace{C\frac{2^{n(N+p)}}{\rho^{sp}}(\xi M)^{p}\left(2+\frac{1}{\theta}(\xi M)^{2-p}\right)|A_n|}_{:= G},
\end{multline}
where we have denoted $G:= \frac{\Gamma (\xi M)^p}{\rho^{sp}}|A_n|$ and $\Gamma:=C 2^{n(N+p)}\left(2+\frac{1}{\theta}(\xi M)^{2-p}\right)$.

Now we can make use of \cref{eq3.16} to further estimate the terms appearing on the right hand side of \cref{eq3.10} to get
\begin{equation}\label{eq3.17}
	\begin{array}{rcl}
		\frac{(1-a)\xi M}{2^n} |A_{n+1}| & \leq & C \left( \Gamma (\xi M)^p|A_n|\right)^{\frac{N}{p(N+2s)}} \lbr \frac{\Gamma(\xi M)^p}{\rho^{N+sp}}|A_n|\rbr^{\frac{s}{N+2s}}|A_n|^{\frac{p(N+2s)-N}{p(N+2s)}}.
	\end{array}
\end{equation}
We now divide \cref{eq3.17} by $|Q_{n+1}|$ noting that $|Q_n| \approx |Q_{n+1}|$ and denote $Y_n := \frac{|A_n|}{|Q_n|}$ from which we get
\[
Y_{n+1}\leq \frac{Cd^n}{(1-a)}\left(\frac{\theta}{(\xi M)^{2-p}}\right)^{\frac{s}{N+2s}}\left(2+\frac{(\xi M)^{2-p}}{\theta}\right)^{\frac{N+sp}{p(N+2s)}}Y_n^{1+\frac{s}{N+2s}},
\]
where $d = d(N,s,p,\La) > 1$ and $C = C(N,s,p,\La)>0$ depends only on the data. By \cref{geo_con}, we see that  $Y_{\infty} = 0$ provided
\begin{equation}
	\label{eq3.18}
Y_0 \leq \lbr\frac{C}{1-a}\rbr^{-\frac{N+2s}{s}}d^{-\left(\frac{N+2s}{s}\right)^2}\frac{\left(\frac{(\xi M)^{2-p}}{\theta}\right)}{\left(2+\frac{(\xi M)^{2-p}}{\theta}\right)^{\frac{N+sp}{sp}}} =: \nu_{-},
\end{equation}
which completes the proof of the lemma.
\end{proof}

\begin{remark}\label{remark3.10}
In the conclusion of \cref{lem:deG}, we can replace $\tail((u-\mu)_{-};x_0,\rho,(t_0-\theta (2\rho)^{sp},t_0]) > \xi M$ with $\tail((u-\mu)_{-};x_0,2\rho,(t_0-\theta (2\rho)^{sp},t_0]) > \xi M$. 

This requires a finer refinement for the estimate of $III$ appearing in \cref{eq3.11} which proceeds as follows:
\begin{align*}
	\underset{{-\theta\rho_n^{sp} < t < 0}}{\esssup}\int_{\RR^N \setminus B_{\rho_n}}\hspace*{-15pt}\frac{(u-b_n)_{-}^{p-1}(y,t)}{|y|^{N+sp}}\,dy = 	\underset{{-\theta\rho_n^{sp} < t < 0}}{\esssup}\lbr \int_{\RR^N \setminus B_{2\rho}}\frac{(u-b_n)_{-}^{p-1}(y,t)}{|y|^{N+sp}}\,dy  + 	\int_{B_{2\rho} \setminus B_{\rho_n}}\frac{(u-b_n)_{-}^{p-1}(y,t)}{|y|^{N+sp}}\,dy \rbr.
\end{align*}
The first term appearing on the right hand side of the previous display can be estimated analogously to \cref{eq315.15} in terms of  $\tail((u-\mu)_{-};x_0,2\rho,(t_0-\theta (2\rho)^{sp},t_0])$. 

In order to estimate the second term, we note that $(u-b_n)_- \leq (\xi M)$ on $B_{2\rho} \times (-\theta\rho_n^{sp},0]$ and $|y| \geq \rho_n \geq \rho$ on $B_{2\rho} \setminus B_{\rho_n}$ from which we get
\[
\underset{{-\theta\rho_n^{sp} < t < 0}}{\esssup}\int_{B_{2\rho} \setminus B_{\rho_n}}\frac{(u-b_n)_{-}^{p-1}(y,t)}{|y|^{N+sp}}\,dy  \leq  C\frac{(\xi M)^{p-1}}{\rho^{N+sp}} |B_{2\rho}\setminus B_{\rho_n}| \leq C\frac{(\xi M)^{p-1}}{\rho^{sp}}, 
\]
which implies \cref{eq3.15} holds for this term. 
\end{remark}

\subsection{De Giorgi lemma: Forward in time  version}
Let $u$ denote a nonnegative, local, weak supersolution of \cref{maineq} in $E_T$ and let $B_R\times I \subset E_T$ denote our reference cylinder. Suppose that we have the following information at a time level $t_0$:
\begin{equation}\label{eq:deGleminitdata}
u(x,t_0) \geq \xi M \qquad \text{ for a.e. } x \in B_{2\rho}(y),    
\end{equation}
for some $M>0$ and $\xi \in (0,1]$. Then in the energy inequality from \cref{energyest},  for any level $k \leq \xi M$ over $[(x_0,t_0)+Q_{2\rho}^{+}(\theta)]$, the first term on the right hand side of \cref{energyest}  over $B_{2\rho}\times\{t_0\}$ vanishes.

Moreover, taking a test function independent of time also kills the integral on the right involving the time derivative of the test function. Therefore we may repeat the same arguments as in \cref{lem:deG} for $(u-\xi_nM )_{-}$ over the cylinders $Q_n^+$ where
\[
\xi_n = a\xi + \frac{1-a}{2^n}\xi \quad \text{and} \quad  Q_n^+ = B_n \times (0,\theta(2\rho)^{sp}].
\]
Denoting $Y_n' = \frac{|[u<\xi_nM]\cap Q_n^+|}{|Q_n^+|}$, we see that if $\tail(u_{-};x_0,2\rho,(t_0,t_0+\theta (2\rho)^{sp}]) \leq \xi M$ is satisfied, then the following estimate holds:
\[
Y_{n+1}' \leq \frac{Cb^n}{(1-a)^{(N+2s)\frac{p}{N}}}\left(\frac{\theta}{(\xi M)^{2-p}}\right)^{\frac{sp}{N}}Y_n'^{1+\frac{sp}{N}}.
\]

Thus, by \cref{geo_con}, we have $Y_{\infty}' = 0$ if
\begin{equation}\label{eq:initdatadeG}
    Y_0' \leq \nu_0\left(\frac{(\xi M)^{2-p}}{\theta}\right)^s =: \tilde{\nu}.
\end{equation}
for a constant $\nu_0 \in (0,1)$ depending only on $a$ and data. The foregoing discussion leads to the following variant of the De Giorgi lemma.
\begin{lemma}\label{lem:deGdata}
Let $u$ denote a nonnegative, local, weak supersolution to \cref{maineq} in $E_T$. Let $M$ and $\xi$ be positive numbers and  suppose \cref{eq:deGleminitdata} holds at time $t=t_0$ and \cref{eq:initdatadeG} is satisfied for $\tilde{\nu} = \tilde{\nu}(\xi,M,\theta,a,N,\La)$ as determined.  Then one of the following two alternatives hold:
\[
\tail(u_{-};x_0,2\rho,(t_0,t_0+\theta (2\rho)^{sp}]) > \xi M \qquad \text{OR} \qquad u \geq a\xi M \,\, \text{ a.e. in } B_{\rho}(x_0)\times (t_0,t_0+\theta(2\rho)^{sp}].
\]
\end{lemma}

\section{Qualitative expansion of positivity in time}\label{exptime}
 In this section, we shall prove a general expansion of positivity estimate that will be used in \cref{sec6} and \cref{sec7}.
\begin{lemma}\label{expandintime}
Let $u$ denote a nonnegative, local, weak supersolution to \cref{maineq} in $E_T$. Assume that for some $(x_0,t_0) \in E_T$ and some $\rho>0$, $M > 0$ and $\al \in (0,1)$, the following hypothesis is satisfied: 
\[
|[u(\cdot,t_0) \geq M]\cap B_{\rho}(x_0)| \geq \alpha|B_{\rho}(x_0)|,
\]
then  there exists $\delta$ and $\epsilon$ in $(0,1)$ depending on $\{N,s,p,\La,\alpha\}$ such that either
\[
\tail(u_{-};x_0,\rho,(t_0,t_0+\delta\rho^{sp}M^{2-p})) > M,
\]
or
\[
|[u(\cdot,t) \geq \epsilon M]\cap B_{\rho}(x_0)| \geq \frac{1}{2}\alpha|B_{\rho}(x_0)|,
\]
holds for all $t \in (t_0,t_0+\delta\rho^{sp}M^{2-p}]$. 
\end{lemma}
\begin{proof}
Without loss of generality, we will assume that $(x_0,t_0)=(0,0)$. For $k>0$ and $t>0$, we set
\[
A_{k,\rho}(t) := [u(\cdot,t) < k] \cap B_{\rho},
\]
then the hypothesis of the lemma can be restated as:
\begin{equation}\label{eq41.1}
|A_{M,\rho}(0)|\leq (1-\alpha)|B_{\rho}|.
\end{equation}
We consider the energy estimate from \cref{energyest} for $(u-M)_{-}$ over the cylinder $B_{\rho}\times(0,\theta\rho^{sp}]$ where $\theta>0$ will be chosen later. Note that $(u-M)_{-} \leq M$ in $B_{\rho}$ because $u$ is nonnegative in $E_T$. For $\sigma \in (0,1/8]$ to be chosen later, we take a cutoff function $\zeta = \zeta(x)$, nonnegative such that it is supported in $B_{(1-\tfrac{\sigma}{2})\rho}$, $\zeta = 1$ on $B_{(1-\sigma)\rho}$ and $|\nabla \zeta| \leq \tfrac{2}{\sigma\rho}$ to get
\begin{equation}\label{eq4.1}
	\def\arraystretch{1.6}
	\begin{array}{l}
		\int_{B_{(1-\sigma)\rho}}(u-M)_-^2(x,\bar{t})\zeta^p(x)\,dx \\
		\hspace*{3cm}\leq  \int_{{B_{\rho}}}(u-M)_-^2(x,0)\zeta^p(x)\,dx\\
		\hspace*{4cm}+ C \int_{0}^{\bar{t}}\iint_{{B_\rho}\times {B_\rho}} \max\{(u-M)_-(x,t),(u-M)_-(y,t)\}^p\frac{|\zeta(x)-\zeta(y)|^p}{|x-y|^{N+sp}}\,dx\,dy\,dt \\ 		
		\hspace*{4cm} + C\underbrace{\int_{0}^{\bar{t}}\int_{ {B_\rho}}\zeta^p(x)(u-M)_-(x,t) \lbr \esssup\limits_{\stackrel{t \in (0,\bar{t})}{x \in \spt\zeta}}\int_{{\RR^N\setminus B_{\rho}}} \frac{(u-M)_-(y,t)^{p-1}}{|x-y|^{N+sp}}\,dy \rbr\,dx\,dt}_{:=III},
	\end{array}
\end{equation}
To estimate the tail term we first note that when $x \in B_{(1-\tfrac{\sigma}{2})\rho}$ and $y \in B_{\rho}^c$ then
\[
|y|\leq |x-y|\left(1+\frac{|x|}{|x-y|}\right) \leq \frac{2}{\sigma}|x-y| \quad \Longrightarrow \quad  \frac{1}{|x-y|^{N+sp}} \leq c\sigma^{-(N+sp)}\frac{1}{|y|^{N+sp}}.
\]
Furthermore, noting that $(u-M)_{-} \leq M-u \leq M + u_{-}$, we get
\begin{equation}\label{eq4.2}
	III  \leq \frac{c}{{\sigma}^{N+sp}}\left(\frac{M^{p-1}}{\rho^{sp}} + \frac{1}{\rho^{sp}}\tll_{\infty}^{p-1}(u_{-};0,\rho,(0,\bar{t}))\right) M\bar{t} |B_\rho| \overset{\text{hypothesis}}{\leq} \frac{c}{{\sigma}^{N+sp}}\frac{M^{p}}{\rho^{sp}} \bar{t} |B_\rho|.
\end{equation}
Substituting \cref{eq4.2} into \cref{eq4.1}, we get
\begin{equation*}
		\int_{B_{(1-\sigma)\rho}}(u-M)_{-}^2(x,\bar{t}) \,dx  \leq   M^2 |A_{M,\rho}(0)| + \frac{C M^p}{\sigma^p \rho^{sp}}\bar{t}|B_{\rho}| + \frac{c}{{\sigma}^{N+sp}}\frac{M^{p}}{\rho^{sp}} \bar{t} |B_\rho|.
\end{equation*}
Now restricting $\bar{t} \leq \delta \rho^{sp} M^{2-p}$ for some $\delta \in (0,1)$ to be chosen, we get
\begin{equation*}
		\begin{array}{rcl}
	\int_{B_{(1-\sigma)\rho}}(u-M)_{-}^2(x,\bar{t}) \,dx  & \leq &   M^2 |A_{M,\rho}(0)| + \delta \frac{C M^2}{\sigma^p}|B_{\rho}| + \frac{c}{{\sigma}^{N+sp}}\delta M^2 |B_\rho| \\
	& \overset{\cref{eq41.1}}{\leq} & M^2 (1-\alpha)|B_\rho| + \delta \frac{C M^2}{\sigma^p}|B_{\rho}| + \frac{c}{{\sigma}^{N+sp}}\delta M^2 |B_\rho| 
		\end{array}
\end{equation*}

On the other hand, we have
\begin{align*}
	\int_{B_{(1-\sigma)\rho}}(u-M)_{-}^2(x,t) \geq \int_{B_{(1-\sigma)\rho}\cap [u<\epsilon M]}(u-M)_{-}^2(x,t)
	\geq M^2(1-\epsilon)^2|A_{\epsilon M,(1-\sigma)\rho}(t)|,
\end{align*}
where $\epsilon \in (0,1)$ will be chosen below depending only on $\alpha$. Next, we note that
\[
|A_{\epsilon M,\rho}(t)| \leq |A_{\epsilon M,(1-\sigma)\rho}(t)| + |B_{\rho}-B_{(1-\sigma)\rho}| \leq |A_{\epsilon M,(1-\sigma)\rho}(t)| + N\sigma|B_{\rho}|. 
\]
Thus, combining everything, we have
\[
|A_{\epsilon M,\rho}(t)| \leq \frac{1}{(1-\epsilon)^2}\left((1-\alpha)+\delta \frac{\bar{C}}{\sigma^{N+sp}} + (1-\epsilon)^2N\sigma \right)|B_{\rho}|.
\]
Now we choose 
\begin{equation*}
	\sigma = \frac{\alpha}{8N}, \qquad \epsilon \leq 1 - \frac{\sqrt{1-\frac{3}{4}\alpha}}{\sqrt{1-\frac{1}{2}\alpha}} \Longrightarrow \frac{1-\alpha}{(1-\epsilon)^2} \leq 1- \frac{3\alpha}{4}, \quad \text{and} \quad  \delta =  \frac{\alpha \sigma^{N+sp} (1-\epsilon)^2}{8\bar{C}},
\end{equation*}
to get the desired conclusion.
\end{proof}

\section{Expansion of positivity for Nonlocal Degenerate Equations}\label{sec6}

We will assume that $u$ is a nonnegative, local, weak supersolution to \cref{maineq} in $E_T$ and  $p>2$. For $(x_0,t_0) \in E_T$ and some given positive number $M$ we consider the cylinder
\[
B_{8\rho}(x_0) \times \left(t_0,t_0+\frac{b^{p-2}}{(\eta M)^{p-2}}\delta\rho^{sp}\right] \subset B_{R}(x_0)\times [t_0-R^{sp},t_0+R^{sp}] \subset E_T,
\]
where $B_{R}(x_0)\times [t_0-R^{sp},t_0+R^{sp}] = B_R(x_0) \times I$ is our reference cylinder. The constants $b,\eta,\delta$ are constants given by \cref{prop:degexp} and $\rho > 0$ is chosen small enough. 
\begin{proposition}\label{prop:degexp}
Assume that for some $(x_0,t_0) \in E_T$, $\rho > 0$, $M>0$ and some $\alpha \in (0,1)$ the following assumption is satisfied:
\begin{equation}\label{eq5.1}
    |[u(\cdot,t_0) \geq M] \cap B_{\rho}(x_0)| \geq \alpha |B_{\rho}(x_0)|.
\end{equation}
 Then there exist constants $\eta, \delta, \sigma' \in (0,1)$ and $b>0$ depending only on the data and $\alpha$ such that either
\[
\tail\left(u_{-};x_0,\rho,\left(t_0,t_0+\frac{b^{p-2}}{(\eta M)^{p-2}}\delta\rho^{sp}\right)\right) > \eta M \quad \text{OR} \quad  u(\cdot,t) \geq \eta M \,\, \text{a.e. in } B_{2\rho}(x_0),
\]
holds for all times
\begin{equation*}
    t_0 + \frac{b^{p-2}}{(\eta M)^{p-2}}\sigma'\delta\rho^{sp} \leq t \leq t_0 + \frac{b^{p-2}}{(\eta M)^{p-2}}\delta\rho^{sp}. 
\end{equation*}
\end{proposition}
\begin{proof} Without loss of generality, we will assume that $(x_0,t_0) = (0,0)$. The proof of the proposition is split into the following steps:

\begin{description}[leftmargin=*]
\descitem{Step 1}{Step 1}\textit{Changing the time variable.}   For all $\sigma \leq 1$, it is easy to see that \cref{eq5.1} implies 
\begin{equation*}
	|[u(\cdot,0) \geq \sigma M] \cap B_{\rho}| \geq \alpha |B_{\rho}| \qquad \forall \ \sigma \leq 1.
\end{equation*}
For $\tau \geq 0$, let us also set 
\begin{equation*}
	\sigma_{\tau} := \exp\left(-\frac{\tau}{p-2}\right) \leq 1,
\end{equation*}
Assuming 
\begin{equation}\label{tail_est_1}
	\tail(u_{-};0,\rho,J_1) \leq \sigma_{\tau} M \qquad \text{where} \quad J_1 := (0,\delta\rho^{sp}(\sigma_{\tau}M)^{2-p}),
\end{equation}
then we can apply \cref{expandintime} to get
\begin{equation}\label{eq5.4}
	\left|\left[u\lbr\cdot,\frac{\delta\rho^{sp}}{(\sigma_{\tau} M)^{p-2}}\rbr \geq \epsilon\sigma_{\tau} M\right] \cap B_{\rho}\right|  \geq \frac{1}{2}\alpha |B_{\rho}|,
\end{equation}
for universal constants $\epsilon,\delta$ (depending only on $p,s,N,\La,\alpha$) in $(0,1)$. In particular, we can rewrite \cref{eq5.4} as 
\begin{equation}\label{eq5.4a}
\left|\left[u\left(\cdot,\frac{e^{\tau}}{ M^{p-2}}\delta\rho^{sp}\right) \geq \epsilon M\sigma_{\tau}\right] \cap B_{\rho}\right| \geq \frac12\alpha |B_{\rho}|.
\end{equation}
 As in \cite[Section 4.2]{dibenedettoHarnackInequalityDegenerate2012}, let us perform the change of variable
\begin{equation}
	w(x,\tau) := \frac{1}{\sigma_{\tau}M}(\delta\rho^{sp})^{\frac{1}{p-2}}u\left(x,\frac{e^{\tau}}{M^{p-2}}\delta\rho^{sp}\right),
\end{equation}
then \cref{eq5.4a} written  in terms of $w$ translates to 
\begin{equation}\label{eq5.5}
	|[w(\cdot,\tau) \geq b_0] \cap B_{\rho}| \geq \frac12 \alpha |B_{\rho}| \quad \text{for all} \ \tau >0,
\end{equation}
where
\begin{equation}\label{eq:k0}
	b_0  := \epsilon(\delta\rho^{sp})^{\frac{1}{p-2}}.
\end{equation}
We can further rewrite \cref{eq5.5} as 
\begin{equation}\label{timeexpdeg}
	|B_{4\rho} \setminus [w(\cdot,\tau) <  b_0]| \geq \frac{1}{2}\alpha 4^{-N} |B_{4\rho}| \qquad \text{for all} \ \tau > 0.
\end{equation}


\descitem{Step 2}{Step 2}{\it Relating $w$ to the evolution equation.} Since $u \geq 0$ in $E_T$, by formal calculations, we have
\begin{align*}
    w_{\tau} &= \left(\frac{1}{\sigma_{\tau}M}(\delta\rho^{sp})^{\frac{1}{p-2}}\right)^{p-1}u_t + \frac{1}{p-2}\frac{1}{M\sigma_{\tau}}(\delta\rho^{sp})^{\frac{1}{p-2}}u
    \\
    &\geq -\left(\frac{1}{\sigma_{\tau}M}(\delta\rho^{sp})^{\frac{1}{p-2}}\right)^{p-1}Lu
    \\
    &\geq - L_1w,
\end{align*}
in $E \times \bb{R}_{+}$ where 
\[
\left\{
\begin{array}{l}
L_1\varphi(x,\tau) = P.V. \int_{\bb{R}^N}B_1(x,y,t)J_p(\varphi(x,\tau)-\varphi(y,\tau)) dy, \\
K_1(x,y,\tau) = K\left(x,y,\frac{e^{\tau}}{M^{p-2}}\delta\rho^{sp}\right), \\
\frac{\Lambda^{-1}}{|x-y|^{N+sp}} \leq K_1(x,y,\tau) \leq 
\frac{\Lambda}{|x-y|^{N+sp}}.
\end{array}
\right.
\]
 The formal calculation can be made rigorous by appealing to the weak formulation and the energy estimates for $u$ can be transferred to energy estimates for $w$ by change of variable.

For any level $k \in \bb{R}$, the energy estimate from \cref{energyest} for $(w-k)_{-}$ in $Q_{8\rho}^{+}(\theta)$ yields
\begin{multline}\label{eq:wcaccio}
\int_{\theta(4\rho)^{sp}}^{\theta(8\rho)^{sp}}\int_{B_{4\rho}}(w-k)_{-}(x,\tau)\int_{B_{4\rho}}\frac{|(w-k)_{+}(y,\tau)|^{p-1}}{|x-y|^{N+sp}}\,dx\,dy\,d\tau
\\
\leq 
C \int_0^{\theta(8\rho)^{sp}}\int_{B_{8\rho}}\int_{B_{8\rho}} \max\{(w-k)_{-}(x,\tau),(w-k)_{-}(y,\tau)\}^{p}\frac{|\zeta(x,\tau)-\zeta(y,\tau)|^p}{|x-y|^{N+sp}}\,dx\,dy\,d\tau 
    \\ 
    + C \int_0^{\theta(8\rho)^{sp}}\int_{B_{8\rho}} (w-k)_{-}^2(x,\tau)|\de_{\tau} \zeta(x,\tau)| \,dx\,d\tau 
     \\
   +C\int_0^{\theta(8\rho)^{sp}}\int_{B_{8\rho}} (w-k)_{-}(x,\tau)\zeta(x,\tau) \lbr \underset{\tau \in (0,\theta(8\rho)^{sp})}{\esssup}\int_{B_{8\rho}^c}\frac{(w-k)_{-}^{p-1}(y,\tau)}{|x-y|^{N+sp}}\,dy\rbr\,dx\,d\tau,
\end{multline}
for any non-negative piecewise smooth cutoff vanishing on the parabolic boundary of $Q_{8\rho}^{+}(\theta)$. In particular, we  choose $\zeta \equiv 1$ on $ 
{Q}_{4\rho}(\theta) = B_{4\rho} \times ((4\rho)^{sp}\theta,(8\rho)^{sp}\theta]
$
with  $\spt(\zeta) \subset B_{6\rho}\times(0,(8\rho)^{sp}\theta]$ and satisfying
\[
|\nabla \zeta| \leq \frac{1}{2\rho} \qquad \text{and} \qquad  |\zeta_{\tau}| \leq \frac{1}{\theta(4\rho)^{sp}}.
\]
\descitem{Step 3}{Step 3}\textit{Shrinking lemma for $w$.} We claim that for every $\nu>0$ there exist $\epsilon_{\nu} \in (0,1)$ depending only on the data and $\alpha$, and $\theta = \theta(b_0,\epsilon_{\nu}) > 0$ depending only on $b_0$ and $\epsilon_{\nu}$ such that if
\begin{equation}\label{tail_est_2}
	\left\{\begin{array}{l}
\tail(u_{-};0,8\rho,J_2) \leq  2\epsilon_{\nu}\epsilon M\exp\left(-\frac{8^{sp}}{(p-2)(2\epsilon_{\nu}\epsilon)^{p-2}\delta}\right),\\
J_2 := \left(M^{2-p}\delta\rho^{sp},\exp\left(\frac{8^{sp}}{(2\epsilon_{\nu}\epsilon)^{p-2}\delta}\right)M^{2-p}\delta \rho^{sp}\right],
\end{array}\right.
\end{equation}

then the following conclusion holds:
\[
|[w < \epsilon_{\nu}B_0]\cap\mathcal{Q}_{4\rho}(\theta)| \leq \nu|\mathcal{Q}_{4\rho}(\theta)|.
\]
\begin{proof}[Proof of Claim:] In \cref{eq:wcaccio} we work with the levels $b_j$ and the parameter $\theta$ as follows
\[
b_j  = \frac{1}{2^j}b_0 \quad \text{ for } \quad 1 \leq j \leq j_{\ast} \qquad \text{ and } \qquad \theta = b_{j_{\ast}}^{2-p},
\]
where $b_0$ is given by \cref{eq:k0} and ${j_{\ast}}$ is to be chosen in \cref{eq:j*deg}.  We first note that since $u \geq 0$ in $Q_{8\rho}^{+}(\theta)$, all the local integrals on the right in \cref{eq:wcaccio} can be estimated using
\[
(w-b_j)_{-} \leq b_j \qquad \text{on} \ Q_{8\rho}^+(\theta).
\]
Thus,  by our choice of the test function and \cref{eq:wcaccio}, we get
\begin{align*}
\int_{\theta(4\rho)^{sp}}^{\theta(8\rho)^{sp}}\int_{B_{4\rho}}(w-b_j)_{-}(x,\tau)&\int_{B_{4\rho}}\frac{|(w-b_j)_{+}(y,\tau)|^{p-1}}{|x-y|^{N+sp}}\,dy\,dx\,d\tau\ 
\\
&\leq C\lbr \frac{b_{j}^p}{\rho^{sp}} + C \frac{b_j^2}{\theta \rho^{sp}} + b_j \underset{\tau \in (0,\theta(8\rho)^{sp})}{\esssup}\int_{B_{8\rho}^c}\frac{(w-b_j)_{-}^{p-1}(y,\tau)}{|y|^{N+sp}}\,dy\rbr |{Q}_{4\rho}(\theta)|.
\end{align*}
From  \cref{sec:tail}, suppose the following holds:
\begin{equation}\label{eq5.11}
\tll_{\infty}^{p-1}(w_{-};0,8\rho,(0,\theta(8\rho)^{sp})) \leq b_{j_{*}}^{p-1},
\end{equation}
  then we have 
\[
\underset{\tau \in (0,\theta(8\rho)^{sp})}{\esssup}\int_{B_{8\rho}^c}\frac{(w-b_j)_{-}^{p-1}(y,\tau)}{|y|^{N+sp}}\,dy  \leq C\frac{b_{j}^p}{\rho^{sp}}|{Q}_{4\rho}(\theta)|.
\]

In particular, recalling the choice of $\theta$, for any $1 \leq j \leq j_{\ast}$, we have
\begin{align*}
	\int_{\theta(4\rho)^{sp}}^{\theta(8\rho)^{sp}}\int_{B_{4\rho}}(w-b_j)_{-}(x,\tau)\int_{B_{4\rho}}\frac{|(w-b_j)_{+}(y,\tau)|^{p-1}}{|x-y|^{N+sp}}\,dy\,dx\,d\tau \leq C \frac{b_{j}^p}{\rho^{sp}}  |{Q}_{4\rho}(\theta)|	
	\end{align*}

The estimate from \cref{timeexpdeg} satisfies hypothesis of \cref{lem:shrinking}, thus we get
\[
|[w < b_{j+1}]\cap{Q}_{4\rho}(\theta)| \leq \left(\frac{C}{2^{j}-1}\right)^{p-1}|{Q}_{4\rho}(\theta)|,
\]
for a constant $C = C(N,p,s,\La,\al)>0$. Thus for a given $\nu \in (0,1)$, we  choose $j_*$ such that
\begin{equation}\label{eq:j*deg}
    \left(\frac{C}{2^{j_*}-1}\right)^{p-1}\leq \nu \qquad \text{and} \quad \epsilon_{\nu} := \frac{1}{2^{j_*+1}}.
\end{equation}

To conclude, we rewrite the Tail alternative in terms of $u_-$ as follows. 
\begin{equation*}
	\begin{array}{rcl}
    \underset{\tau \in (0,\theta(8\rho)^{sp})}{\esssup}\int_{B_{8\rho}^c}\frac{w_{-}^{p-1}(y,\tau)}{|y|^{N+sp}}\,dy &=& (\delta\rho^{sp})^{\frac{p-1}{p-2}}\frac{1}{M^{p-1}}\underset{\tau \in (0,\theta(8\rho)^{sp})}{\esssup}\int_{B_{8\rho}^c}\frac{u_{-}^{p-1}(y,e^{\tau}M^{2-p}\delta\rho^{sp})}{|y|^{N+sp}}e^{\tau\frac{p-1}{p-2}}\,dy
    \\
    &\leq & (\delta\rho^{sp})^{\frac{p-1}{p-2}}\frac{\exp\left(\theta(8\rho)^{sp}\frac{p-1}{p-2}\right)}{M^{p-1}}\underset{t \in J_2}{\esssup}\int_{B_{8\rho}^c}\frac{u_{-}^{p-1}(y,t)}{|y|^{N+sp}}\,dy,
    \end{array}
\end{equation*}
where 
\[
J_2 = \left(M^{2-p}\delta\rho^{sp},\exp((8\rho)^{sp}\theta)M^{2-p}\delta \rho^{sp}\right].
\]
Thus, \cref{eq5.11}  is verified provided
\[
\tail(u_{-};0,8\rho,J_2) \leq \frac{b_{j_{*}}}{(\delta\rho^{sp})^{\frac{1}{p-2}}}\frac{M}{\exp\left(\theta(8\rho)^{sp}\frac{1}{p-2}\right)}.
\]
Recalling that
\[
b_{j_{*}} = \frac{b_0}{2^{j_{*}}} = \frac{\epsilon(\delta\rho^{sp})^{\frac{1}{p-2}}}{2^{j_{*}}} = 2\epsilon_{\nu}\epsilon(\delta\rho^{sp})^{\frac{1}{p-2}} = \theta^{\frac{1}{2-p}},
\]
we get
\[
\exp\left(\frac{\theta(8\rho)^{sp}}{p-2}\right) = \exp\left(\frac{8^{sp}}{(p-2)(2\epsilon_{\nu}\epsilon)^{p-2}\delta}\right),
\]
and
\[
\frac{b_{j_{*}}}{(\delta\rho^{sp})^{\frac{1}{p-2}}}\frac{M}{\exp\left(\theta(8\rho)^{sp}\frac{1}{p-2}\right)} = 2\epsilon_{\nu}\epsilon M\exp\left(-\frac{8^{sp}}{(p-2)(2\epsilon_{\nu}\epsilon)^{p-2}\delta}\right),
\]
which completes the proof of the claim.
\end{proof}
\descitem{Step 4}{Step 4}\textit{Expansion of positivity for $w$.} We claim that there exists a $\nu = \nu(N,s,p,\La,\alpha) \in (0,1)$  such that either
\[
\tail(u_{-};0,4\rho,J_2) >  2\epsilon_{\nu}\epsilon M\exp\left(-\frac{8^{sp}}{(p-2)(2\epsilon_{\nu}\epsilon)^{p-2}\delta}\right)
\]
or
\[
w(\cdot,\tau) \geq \frac{1}{2}\epsilon_{\nu}b_0 \qquad \text{ a.e. in } B_{2\rho}\times \left(\frac{(8\rho)^{sp}-(2\rho)^{sp}}{(2\epsilon_{\nu}b_0)^{p-2}},\frac{(8\rho)^{sp}}{(2\epsilon_{\nu}b_0)^{p-2}}\right]
\]
holds, where $\epsilon_{\nu}$ is the number corresponding to $\nu$ in \descref{Step 3}{Step 3} (see \cref{eq:j*deg}) and $J_2$ is as defined in \cref{tail_est_2}.
\begin{proof}[Proof of the claim] We apply \cref{lem:deG} along with \cref{remark3.10} to $w$ over the cylinder
\[
{Q}_{4\rho}(\theta) = (0,\tau^*)+Q_{4\rho}^{-}(\theta) \qquad \text{ for } \tau_{*} = \theta(8\rho)^{sp}.
\]
We work with $\epsilon_{\nu}b_0$ instead of $\xi\omega$, $a = 1/2$ and $\mu_{-} \geq 0$ is ignored (since we assume $u$ is non-negative subsolution) to get that either
\begin{equation}\label{eq:tail3deg}
    \tail((w)_{-};0,4\rho,(\tau_{*}-\theta(4\rho)^{sp},\tau_{*}]) > \epsilon_{\nu}b_0,
\end{equation}
holds or 
\[
w(x,\tau) \geq \frac{1}{2}\epsilon_{\nu}b_0 \qquad \text{ for a.e. } \qquad (x,\tau) \in [(0,\tau^*)+Q_{2\rho}^{-}(\theta)],
\]
holds provided
\[
\frac{\left|[w<\epsilon_{\nu}b_0]\cap {Q}_{4\rho}(\theta)\right|}{|{Q}_{4\rho}(\theta)|} \leq \frac{1}{C}\left(\frac{1}{2}\right)^{N+2s}\frac{[\theta(\epsilon_{\nu}b_0)^{p-2}]^{\frac{N}{sp}}}{[1+2\theta(\epsilon_{\nu}b_0)^{p-2}]^{\frac{N+sp}{sp}}} = \frac{1}{C}\left(\frac{1}{2}\right)^{N+2s} \frac{1}{3^\frac{N+sp}{sp}} =\nu,
\]
for a universal constant $C>1$. With this choice of $\nu$, we find $\epsilon_{\nu}$ from \descref{Step 3}{Step 3}  and therefore $\theta = \frac{(2\epsilon_\nu\epsilon)^{2-p}}{\delta \rho^{sp}}$ quantitatively. 
The tail alternative \cref{eq:tail3deg} can be reformulated in terms of $u$ as in \descref{Step 3}{Step 3} to get
\begin{equation}\label{tail_est_3}	\left\{\begin{array}{l}
\tail(u_{-};0,4\rho,J_3) \leq  2\epsilon_{\nu}\epsilon M\exp\left(-\frac{8^{sp}}{(p-2)(2\epsilon_{\nu}\epsilon)^{p-2}\delta}\right),\\ 
J_3 := (\exp((8\rho)^{sp}\theta-(4\rho)^{sp}\theta)M^{2-p}\delta\rho^{sp},\exp((8\rho)^{sp}\theta)M^{2-p}\delta \rho^{sp}].
\end{array}\right.
\end{equation}
Since $J_2 \supset J_3$, the claim follows.
\end{proof}
\descitem{Step 5}{Step 5} {\it Expanding the positivity for $u$.} Assume that the Tail alternatives from \cref{tail_est_1},\cref{tail_est_2} and \cref{tail_est_3} hold . As $\tau$ ranges over the interval in \descref{Step 4}{Step 4}, we see that $e^{\tfrac{\tau}{p-2}}$ ranges over the interval
\[
k_1 := \exp\left(\frac{8^{sp}-2^{sp}}{(p-2)[2\epsilon_{\nu}\epsilon \delta^{\frac{1}{p-2}}]^{p-2}}\right) \leq f(\tau) \leq \exp\left(\frac{8^{sp}}{(p-2)[2\epsilon_{\nu}\epsilon \delta^{\frac{1}{p-2}}]^{p-2}}\right) =: k_2.
\]
Rewriting the conclusion of \descref{Step 4}{Step 4} in terms of $u$, we get that 
\begin{equation}\label{def_eta}
u(x,t) \geq \frac{\epsilon_{\nu}\epsilon M}{2k_2} =: \eta M \qquad \text{ for a.e. } \quad x \in B_{2\rho},
\end{equation}
holds for all times 
\[
 \frac{b^{p-2}}{(\eta M)^{p-2}}\sigma\delta\rho^{sp} \leq t \leq \frac{b^{p-2}}{(\eta M)^{p-2}}\delta\rho^{sp}, 
\]
where $b = b(N,s,p,\La,\al) > 0$  and $\sigma' = \sigma(N,s,p,\La,\al) \in (0,1)$. In fact, we have
\[
b = \frac{\epsilon_{\nu}\epsilon}{2} \quad \text{and} \qquad \text{and} \quad \sigma' = \left(\frac{k_1}{k_2}\right)^{p-2}.
\]
\descitem{Step 6}{Step 6} {\it Pulling the Tail alternatives together.}  We need the following Tail alternatives going from \descref{Step 1}{Step 1} and \descref{Step 3}{Step 3} noting that \cref{tail_est_2} implies \cref{tail_est_3}. In particular, we recall
\[
\begin{array}{rcl}
\tail(u_{-};0,\rho,J_1) &\leq&\sigma_{\tau} M,\\
\tail(u_{-};0,4\rho,J_2) &\leq &  2\epsilon_{\nu}\epsilon M\exp\left(-\frac{8^{sp}}{(p-2)(2\epsilon_{\nu}\epsilon)^{p-2}\delta}\right), \\
\tail(u_{-};0,8\rho,J_2) &\leq &  2\epsilon_{\nu}\epsilon M\exp\left(-\frac{8^{sp}}{(p-2)(2\epsilon_{\nu}\epsilon)^{p-2}\delta}\right),
\end{array}
\]
where we have taken 
\[
\begin{array}{rcl}
J_1 &=& (0,\delta\rho^{sp}(\sigma_{\tau}M)^{2-p}),\\
J_2 &=& \left(M^{2-p}\delta\rho^{sp},\exp\left(\frac{8^{sp}}{(2\epsilon_{\nu}\epsilon)^{p-2}\delta}\right)M^{2-p}\delta \rho^{sp}\right].
\end{array}
\]

Since we have $\tau\leq\theta(8\rho)^{sp}$ which implies the following inclusion holds:
\[
J_1 \subset (0,\delta\rho^{sp}M^{2-p}\exp(\theta(8\rho)^{sp})) = (0,\delta\rho^{sp}M^{2-p}\exp(8^{sp}\delta^{-1}(2\epsilon_{\nu}\epsilon)^{2-p})),
\]
and $\sigma_{\tau} \leq  \exp\left(-\frac{8^{sp}}{(p-2)(2\epsilon_{\nu}\epsilon)^{p-2}\delta}\right)$.
Furthermore, we choose $\eta$ such that 
\[
2\epsilon_{\nu}\epsilon M\exp\left(-\frac{8^{sp}}{(p-2)(2\epsilon_{\nu}\epsilon)^{p-2}\delta}\right) = \frac{2\epsilon_{\nu}\epsilon M}{k_2} \overset{\cref{def_eta}}{=} 4\eta M.
\]
Further choosing $\eta \in (0,1)$ small such that by an abuse of notation, denoting  $\eta = \frac{\eta}{8^{sp}}$, b the final Tail alternative which subsumes all others is
\[
\tail(u_{-};0,\rho,J) \leq \eta M \quad \text{where} \quad J := \left(0,\frac{b^{p-2}}{(\eta M)^{p-2}}\delta\rho^{sp}\right],
\]
where we recall $ b = \frac{\epsilon_{\nu}\epsilon}{2}$.
\end{description}
\end{proof}
\begin{remark}
The conclusion of \cref{prop:degexp} can also be written without Tail alternatives as 
\[
u(\cdot,t) \geq \eta M - \tail\left(u_{-};x_0,\rho,\left(t_0,t_0+\frac{b^{p-2}}{(\eta M)^{p-2}}\delta\rho^{sp}\right)\right) \qquad \text{a.e. in } B_{2\rho}(x_0), 
\]
holds for all times
\[
    t_0 + \frac{b^{p-2}}{(\eta M)^{p-2}}\sigma' \delta\rho^{sp} \leq t \leq t_0 + \frac{b^{p-2}}{(\eta M)^{p-2}}\delta\rho^{sp},
\]
because we are working with  nonnegative solutions.
\end{remark}
\begin{remark}
    In the proof of H\"older regularity we use \cref{prop:degexp} with $\al = 1/2$ so that  $\eta, \delta, \sigma' \in (0,1)$ and $b>0$ depend only on the data in the proof. 
\end{remark}
\section{Expansion of positivity for Nonlocal Singular Equations}\label{sec7}

In this section, we will consider $1 < p < 2$ and  assume that $u$ is a nonnegative, local, weak supersolution to \cref{maineq} in $E_T$. For $(x_0,t_0) \in E_T$ and some given positive numbers $M >0$ and $\delta \in (0,1)$, we consider the cylinder
\[
(x_0,t_0) + Q_{16\rho}(\delta M^{2-p}): = B_{16\rho}(x_0) \times (t_0,t_0+\delta M^{2-p}\rho^{sp}] \subset B_{R}(x_0)\times [t_0-R^{sp},t_0+R^{sp}] \subset E_T,
\]
where $B_{R}(x_0)\times [t_0-R^{sp},t_0+R^{sp}] = B_R(x_0) \times I$ is our reference cylinder and $\rho > 0$ is chosen small enough. 
\begin{proposition}\label{prop:sinexp}
Assume that for some $(x_0,t_0) \in E_T$ and some $\rho > 0$, $M>0$ and $\alpha \in (0,1)$,  the following hypothesis holds
\begin{equation*}
    |[u(\cdot,t_0) \geq M] \cap B_{\rho}(x_0)| \geq \alpha |B_{\rho}(x_0)|.
\end{equation*}
 Then there exist constants $\eta,\delta$ and $\varepsilon$ in $(0,1)$ depending only on the data and $\alpha$ such that either
\[
\tail(u_{-};x_0,\rho,(t_0,t_0+\delta\rho^{sp}M^{2-p}))  >  \eta M \qquad  \text{OR} \qquad  u(\cdot,t) \geq \eta M \qquad \text{a.e. in } B_{2\rho}(x_0),
\]
holds for all times
\begin{equation*}
    t_0 + (1-\varepsilon)\delta M^{2-p}\rho^{sp} \leq t \leq t_0 + \delta M^{2-p}\rho^{sp}.
\end{equation*}
\end{proposition}
\begin{proof}Without loss of generality, we will assume that $(x_0,t_0) = (0,0)$.
\begin{description}[leftmargin=*]
\descitem{Step 1}{step1} {\it Changing variables.} From \cref{exptime} we get that there exist $\delta$ and $\epsilon$ in $(0,1)$ depending only on the data and $\alpha$ such that one of the following two possibilities must hold:
\[
\tail(u_{-};0,\rho,(0,\delta\rho^{sp}M^{2-p})) > M,
\]
or 
\begin{equation}\label{eq:timeexpsing}
|[u(\cdot,t) > \epsilon M] \cap B_{\rho}| \geq \frac{1}{2}\alpha|B_{\rho}| \qquad \text{ for all } \qquad t \in (0,\delta M^{2-p}\rho^{sp}].
\end{equation}
Let us now define the change of variables
\begin{equation}\label{eq:covsing}
z = \frac{x}{\rho}, \qquad -e^{-\tau} = \frac{t-\delta M^{2-p}\rho^{sp}}{\delta M^{2-p}\rho^{sp}} \qquad \text{and} \quad v(z,\tau) = \frac{1}{M}u(x,t)e^{\frac{\tau}{2-p}}.
\end{equation}
Then the cylinder $Q_{16\rho}(\delta M^{2-p})$ in $(x,t)$ coordinates is mapped into $B_{16}\times (0,\infty)$ in the $(z,\tau)$ coordinates. By formal calculations (which can be made rigorous using the weak formulation), we have
\begin{align}\label{eq:expeqsing}
    v_{\tau} = \frac{1}{2-p}v + \frac{e^{\frac{p-1}{2-p}\tau}\delta\rho^{sp}}{M^{p-1}}u_t(z\rho,\delta M^{2-p}\rho^{sp}(1-e^{-\tau}))
    = \frac{1}{2-p}v + \delta L_1v,
\end{align}
in $B_{16} \times (0,\infty)$, where 
\[\begin{array}{rcl}
L_1\varphi(z,\tau) & = &  P.V. \int_{\bb{R}^N}K_1(z,z',\tau)J_p(\varphi(z,\tau)-\varphi(z',\tau)) \,dz', \\
K_1(z,z',\tau) &=& \rho^{N+sp}K\left(z\rho,z'\rho,\delta M^{2-p}\rho^{sp}(1-e^{-\tau})\right),
\end{array}
\]
satisfying
\[
\frac{\Lambda^{-1}}{|z-z'|^{N+sp}} \leq K_1(z,z',\tau) \leq 
\frac{\Lambda}{|z-z'|^{N+sp}}.
\]
 We rewrite \cref{eq:timeexpsing} as 
\begin{equation}\label{eq:sing2}
    |[v(\cdot,\tau) \geq \epsilon e^{\frac{\tau}{2-p}}]\cap B_1| \geq \frac{1}{2}\alpha |B_1| \qquad \text{ for all } \tau \in (0,\infty).
\end{equation}
For {constants $\tau_0 > 0$ (chosen in \cref{tauz0}) and $j_{*} \in \NN$ (chosen in \descref{step3}{Step 3})}, we take
\begin{equation}\label{eq:k0sing}
b_0 = \epsilon e^{\frac{\tau_0}{2-p}} \qquad \text{and} \quad  b_j = b_02^{-j}\ \  \text{ for } 1 \leq j \leq j_{*}.
\end{equation}
Then from \cref{eq:sing2}, we see that
\begin{equation}\label{eq:6.4}
    |[v(\cdot,\tau) \geq b_j]\cap B_{8}| \geq \frac{1}{2}\alpha 8^{-N}|B_8| \qquad \text{  holds for all } \tau \in (\tau_0,\infty) \quad \text{and} \quad \forall \,j \in \NN.
\end{equation}

 Consider the cylinders
\begin{equation}\label{6.7:eq}
Q_{\tau_0} = B_8 \times (\tau_0+b_0^{2-p},\tau_0+2b_0^{2-p}) \quad \text{and} \quad 
Q_{\tau_0'} = B_{16} \times (\tau_0, \tau_0+2b_0^{2-p}),
\end{equation}
and a nonnegative, piecewise smooth cutoff in $Q_{\tau_0'}$ of the form $\zeta(z,\tau) = \zeta_1(z)\zeta_2(\tau)$ such that
\[
\begin{array}{ll}
\zeta_1 = \left\{ \begin{array}{lcl} 1 & \text{in} & B_8 \\ 0 & \text{on} & \RR^N\setminus B_{12}\end{array}\right. & |\de \zeta_1| \leq \frac14,\\
\zeta_2 = \left\{ \begin{array}{lcl} 0 & \text{for} & \tau < \tau_0 \\ 1 & \text{for} & \tau \geq \tau_0 + b_0^{2-p} \end{array}\right. & |\de_t \zeta_2| \leq \frac1{b_0^{2-p}}.
\end{array}
\]
From \cref{energyest} applied  for $(v-b_j)_{-}$ over $Q_{\tau_0'}$ with the cutoff  function $\zeta$, we get
\begin{align}\label{eq:vcaccio}
\int_{\tau_0}^{\tau_0 +2b_0^{2-p}}&\int_{B_{16}}(v-b_j)_{-}(z,\tau)\zeta^p(z,\tau)\int_{B_{16}}\frac{(v-b_j)_{+}^{p-1}(z',\tau)}{|z-z'|^{N+sp}}\,dz'\,dz\,d\tau\  \nonumber
\\
&\leq 
C \int_{\tau_0}^{\tau_0 +2b_0^{2-p}}\int_{B_{16}}\int_{B_{16}} \max\{(v-b_j)_{-}(z,\tau),(v-b_j)_{-}(z',\tau)\}^{p}\frac{|\zeta(z,\tau)-\zeta(z',\tau)|^p}{|z-z'|^{N+sp}}\,dz\,dz'\,d\tau \nonumber
    \\ 
     &+ C \int_{\tau_0}^{\tau_0 +2b_0^{2-p}}\int_{B_{16}} (v-b_j)_{-}^2(z,\tau)|\de_\tau \zeta(z,\tau)| \,dz\,d\tau \nonumber
     \\
     &+C \int_{\tau_0}^{\tau_0 +2b_0^{2-p}}\int_{B_{16}} (v-b_j)_{-}(z,\tau)\zeta^p(z,\tau)\left(\underset{\tau \in (\tau_0, \tau_0+2b_0^{2-p})}{\esssup}\int_{B_{16}^c}\frac{(v-b_j)_{-}^{p-1}(z',\tau)}{|z-z'|^{N+sp}}\,dz'\right)\,dz\,d\tau,
\end{align}
where $C = \frac{C}{\delta}$ (note the $\delta$ in \cref{eq:expeqsing}) is a constant depending only on the data and $\delta$. To obtain the above estimate, we tested \cref{eq:expeqsing} with $-(v-b_j)_{-}\zeta^p$ and discarded the nonpositive contribution on the right-hand side from the nonnegative term $\tfrac{1}{2-p}v$.
\descitem{Step 2}{step2} {\it Shrinking lemma for $v$.} For any $j \geq 2$, we claim that either of the following two possibilities hold:
\begin{equation}\label{tailest_sing_2}\left\{
\begin{array}{l}
\tail(u_{-};0,16\rho,J_2') > \exp\left(\tfrac{-2b_0^{2-p}}{2-p}\right)\frac{\epsilon M}{2^j}, \\
J_2' = \left(M^{2-p}\delta\rho^{sp}(1-e^{-\tau_0}),M^{2-p}\delta\rho^{sp}(1-e^{-\tau_0-2b_0^{2-p}})\right),
\end{array}\right.
\end{equation}
or
\[
|[v < b_{j}]\cap Q_{\tau_0}| \leq \nu|Q_{\tau_0}| \qquad \text{where} \quad \nu = \left(\frac{C}{2^{j-1}-1}\right)^{p-1},
\]
for a universal constant $C>0$ depending only on $\alpha$ and data. 

\begin{proof}[Proof of Claim] In \cref{eq:vcaccio}, we work with levels $b_j$ defined by
\[
b_j  = \frac{1}{2^j}b_0 \qquad \text{ for } \qquad j \geq 1,
\]
where $b_0$ is given by \cref{eq:k0sing}. We first note that since $v \geq 0$ in $Q_{\tau_0'}$, we have  $(v-b_j)_{-} \leq b_j$ on $Q_{\tau_0'}$. Hence from our choice of the test function, we get
\begin{multline*}
\int_{\tau_0+b_0^{2-p}}^{\tau_0 +2b_0^{2-p}} \int_{B_{8}}(v-b_j)_{-}(z,\tau)\int_{B_{8}}\frac{|(v-b_j)_{+}(z',\tau)|^{p-1}}{|z-z'|^{N+sp}}\,dz'\,dz\,d\tau \\
\leq C|Q_{\tau_0'}|\lbr b_{j}^p
     +b_{j}\underset{\tau \in (\tau_0, \tau_0+2b_0^{2-p}) }{\esssup}\int_{B_{16}^c}\frac{(v-b_j)_{-}^{p-1}(z',\tau)}{|z'|^{N+sp}}\,dz'\rbr
\end{multline*}
where $C>0$ is a universal constant. We now make use of the Tail estimates from \cref{sec:tail} to get
\[
\int_{\tau_0+b_0^{2-p}}^{\tau_0 +2b_0^{2-p}}\int_{B_{8}}(v-b_j)_{-}(z,\tau)\int_{B_{8}}\frac{|(v-b_j)_{+}(z',\tau)|^{p-1}}{|z-z'|^{N+sp}}\,dz'\,dz\,d\tau\  \leq Cb_{j}^p|Q_{\tau_0'}|,
\]
holds provided the following is satisfied
\begin{equation}\label{eq:Tailsingv}
\tll_{\infty}^{p-1}(v_{-};0,16,(\tau_0, \tau_0+2b_0^{2-p})) \leq b_j^{p-1}.
\end{equation}


We now invoke \cref{lem:shrinking} and \cref{eq:sing2} to conclude that
\[
|[v < b_{j+1}]\cap Q_{\tau_0}| \leq \left(\frac{C}{2^{j}-1}\right)^{p-1}|Q_{\tau_0'}|\overset{\cref{6.7:eq}}{\leq} C \left(\frac{C}{2^{j}-1}\right)^{p-1}|Q_{\tau_0}|,
\]
for a constant $C>0$ depending only on $\alpha$ and data. This proves the claim with the choice
$\nu = \left(\frac{C}{2^{j}-1}\right)^{p-1}.$

We now rewrite the Tail alternative in terms of $u_-$ as follows. 
\begin{equation}\label{Tailestsingu1}
	\begin{array}{rcl}
    \underset{\tau \in (\tau_0, \tau_0+2b_0^{2-p})}{\esssup}\int_{B_{16}^c}\frac{v_{-}^{p-1}(z',\tau) }{|z'|^{N+sp}}\,dz'& \overset{\cref{eq:covsing}}{\leq}& \frac{\exp((\tau_0+2b_0^{2-p})\frac{p-1}{2-p})}{M^{p-1}}\rho^{sp}\left(\underset{t\in J_2'}{\esssup}\int_{B_{16\rho}^c}\frac{u(y,t)^{p-1}_{-}}{|y|^{N+sp}}\,dy\right)
    \\
    &\leq& \frac{\exp((\tau_0+2b_0^{2-p})\frac{p-1}{2-p})}{M^{p-1}}\tll_{\infty}^{p-1}(u_{-};0,16\rho,J_2'),
    \end{array}
\end{equation}
where 
\[
J_2' := \left(M^{2-p}\delta\rho^{sp}(1-e^{-\tau_0}),M^{2-p}\delta\rho^{sp}(1-e^{-\tau_0-2b_0^{2-p}})\right).
\]

Thus, making use of \cref{Tailestsingu1}, we see that \cref{eq:Tailsingv} is satisfied provided
\[
\tail(u_{-};0,16\rho,J_2') \leq  \exp\left(\frac{-\tau_0-2b_0^{2-p}}{2-p}\right)Mb_j.
\]
Recalling \cref{eq:k0sing}, we get
\[
\exp\left(\frac{-\tau_0-2b_0^{2-p}}{2-p}\right)Mb_j = \exp\left(\frac{-2b_0^{2-p}}{2-p}\right)\frac{\epsilon M}{2^j},
\]
which completes the proof of the claim.
\end{proof}

\descitem{Step 3}{step3} {\it Obtaining a good time slice for $v$.}  We claim that there exist $\sigma_0 \in (0,1)$ and $j_{*} \in (1,\infty)$ depending only on the data such that either
\[
\tail(u_{-};0,8\rho,J_2') > \exp\left(\frac{-2b_0^{2-p}}{2-p}\right)\frac{\epsilon M}{2^{j_{*}}},
\]
holds (recall \cref{tailest_sing_2}) or
there exists a time level $\tau_1 \in (\tau_0 + b_0^{2-p}, \tau_0 + 2b_0^{2-p})$ such that the following holds:
\begin{equation}\label{conc2:eq}
v(z,\tau_1) \geq \sigma_0e^{\frac{\tau_0}{2-p}}.
\end{equation}

\begin{proof}[Proof of the claim] Assume momentarily that $j_*$  is fixed and hence $\nu$ has been determined according to \descref{step2}{Step 2}. By increasing $j_*$ if required to not necessarily be an integer, we can further assume without loss of generality that $2^{j_{*}(2-p)}$ is an integer. Next, we subdivide $Q_{\tau_0}$ into  $2^{j_{*}(2-p)}$ cylinders each of length $b_{j_*}^{2-p}$ given by
\[
Q_n = B_8 \times (\tau_0 + b_0^{2-p} + nb_{j_*}^{2-p}, \tau_0 + b_0^{2-p} + (n+1)b_{j_*}^{2-p}) \qquad\mbox{ for } n = 0,1,\ldots,2^{j_{*}(2-p)}-1.
\]
Then from \descref{step2}{Step 2}, either \cref{tailest_sing_2} holds or for at least one of the sub-cylinders, we must have
\[
|[v<b_{j_*}\cap Q_n]| \leq \nu |Q_n|. 
\]
We now apply \cref{lem:deG} to $v$ over $Q_n$ with 
\[
\mu_- = 0, \qquad \xi M = b_{j_*}, \qquad a = \frac{1}{2}, \qquad \text{and} \quad  \theta = b_{j_*}^{2-p},
\]
to obtain 
\begin{equation}\label{conc:eq}
v(z,\tau_0 + b_0^{2-p} + (n+1)b_{j_*}^{2-p}) \geq \frac{1}{2}b_{j_*} \qquad \text{ a.e. } z\in B_4.
\end{equation}
 provided the following two estimates hold
\begin{equation}\label{eq:tailstep3sing}
	\begin{array}{c}
\tail(v_{-};0,8,(\tau_0 + b_0^{2-p} + nb_{j_*}^{2-p}, \tau_0 + b_0^{2-p} + (n+1)b_{j_*}^{2-p})) \leq b_{j_{*}},\\
\frac{|[v<b_{j_*}]\cap Q_n|}{|Q_n|} \overset{\cref{eq3.18}}{\leq} 3^{- \frac{N+sp}{sp}}\overline{\gamma_0}(\text{data}) =: \nu.
\end{array}
\end{equation}

With the choice of $\nu$ from \cref{eq:tailstep3sing}, we can now choose $j_*$ according to \descref{step2}{Step 2} such that \cref{conc:eq} holds. Recalling the definition of $b_0$ from \cref{eq:k0sing}, if we set $\sigma_0 = \epsilon2^{-(j_{*}+1)}$, then we see that \cref{conc2:eq} holds with $\tau_1=\tau_0 + b_0^{2-p} + (n+1)b_{j_*}^{2-p}$.

Let us now rewrite the Tail alternative in \cref{eq:tailstep3sing} in terms of $u$ as in \descref{step2}{Step 2} to get 
\[
\tail(u_{-};0,8\rho,J_2') \leq \exp\left(\frac{-2b_0^{2-p}}{2-p}\right)\frac{\epsilon M}{2^{j_{*}}}.
\]
In particular, the claim holds provided the Tail alternative from \cref{eq:tailstep3sing} is satisfied. \end{proof}
\descitem{Step 4}{step4} {\it Expanding the positivity for $u$.} Recalling the change of variables from \cref{eq:covsing}, we see that if the Tail alternatives in \descref{step2}{Step 2} and \descref{step3}{Step 3} are satisfied, then the following conclusion holds
\[
u(\cdot,t_1) \geq \sigma_0Me^{-\frac{\tau_1-\tau_0}{2-p}} =: M_0 \qquad \text{ in } B_{4\rho},
\]
where $t_1=\delta M^{2-p}\rho^{sp}(1-e^{-\tau_1})$. 

With $\nu_0$ as obtained in \cref{eq:initdatadeG}, let us first choose $\theta = \nu_0^{\frac{1}{s}}M_0^{2-p}$. We now apply \cref{lem:deGdata} (noting that $\tilde{\nu}=1$ in \cref{eq:initdatadeG} so that the smallness assumption  is trivially satisfied) with $M_0$ in place of $M$ and $\xi = 1$ over the cylinder $B_{4\rho} \times (t_1,t_1+\theta(4\rho)^{sp}]$ to see that the following holds
\begin{equation*}
    u(x,t) \geq \frac{1}{2}M_0 \geq \frac{1}{2}\sigma_0\exp\left(-\frac{2}{2-p}e^{\tau_0}\right)M \qquad \text{ for all } (x,t)\in B_{2\rho}\times(t_1,t_1+\nu_0^{\frac{1}{s}}M_0^{2-p}(4\rho)^{sp}),
\end{equation*}
provided $\tail(u_{-};0,4\rho,(t_1,t_1+\theta (4\rho)^{sp}]) \leq M_0$ is satisfied.

We now choose $\tau_0$ satisfying 
\begin{equation}\label{tauz0}
\delta M^{2-p}\rho^{sp}e^{-\tau_1} = \delta M^{2-p}\rho^{sp} - t_1 = \nu_0^{\frac{1}{s}} \sigma_0^{2-p} M^{2-p}(4\rho)^{sp}e^{-(\tau_1-\tau_0)}\iff\tau_0 = \ln\left(\frac{\delta}{4^{sp}\nu_0^{\frac{1}{s}}\sigma_0^{2-p}}\right).
\end{equation}
This determines $\tau_0$ in terms of $\alpha$ and data (note that by increasing $j_{*}$ if necessary we can always ensure that $\tau_0 > 0$).  Finally, setting $\eta := \frac{1}{2}\sigma_0\exp\left(-\frac{2}{2-p}e^{\tau_0}\right)\in (0,1)$ which depends only on $\alpha$ and data, we get that
\[
u(\cdot,t) \geq \eta M \qquad \text{ a.e. in } B_{2\rho} ,
\]
holds for all times
\begin{equation}\label{def_ve}
(1-\varepsilon)\delta M^{2-p}\rho^{sp} \leq t \leq M^{2-p}\delta\rho^{sp} \qquad \text{ where } \varepsilon = \exp(-\tau_0-2e^{\tau_0}). 
\end{equation}
\descitem{Step 5}{step5} {\it Pulling the Tail alternatives together.} We needed the following Tail assumptions going from \descref{step1}{Step 1} to \descref{step4}{Step 4} (recall $j \leq j_{*})$:
\[
\begin{array}{rcl}
\tail(u_{-};0,\rho,(0,\delta\rho^{sp}M^{2-p})) &\leq & M,\\
\tail(u_{-};0,16\rho,J_2') &\leq & \exp\left(\tfrac{-2b_0^{2-p}}{2-p}\right)\frac{\epsilon M}{2^j},\\
\tail(u_{-};0,8\rho,J_2') &\leq & \exp\left(\tfrac{-2b_0^{2-p}}{2-p}\right)\frac{\epsilon M}{2^{j_{*}}},\\
\tail(u_{-};0,4\rho,(t_1,t_1+\theta (4\rho)^{sp}]) &\leq & M_0.
\end{array}
\]

We note that in all the Tail alternatives, $\rho$ is the smallest radius and $(0,\delta\rho^{sp}M^{2-p})$ is the largest time interval. We estimate
\begin{gather*}
\exp\left(\frac{-2b_0^{2-p}}{2-p}\right)\frac{\epsilon M}{2^j} = \exp\left(-\frac{2\epsilon^{2-p}e^{\tau_0}}{2-p}\right)\frac{\epsilon M}{2^j} \geq \exp\left(-\frac{2e^{\tau_0}}{2-p}\right)\frac{\epsilon M}{2^j} = \frac{2\eta}{\sigma_0}\frac{\epsilon M}{2^j} \geq 4\eta M,\\
M_0 \geq \sigma_0\exp\left(-\frac{2}{2-p}e^{\tau_0}\right)M = 2\eta M,
\end{gather*}
where we recall $ \epsilon $ is from \cref{def_ve} and $\sigma_0 = \frac{\epsilon}{2^{j_{*}+1}}$.

Since $\eta\in (0,1)$, we see that all the Tail assumptions are satisfied if we require
\[
\tail(u_{-};0,\rho,(0,\delta\rho^{sp}M^{2-p})) \leq \eta M.
\]
\end{description}
This completes the proof of the proposition.
\end{proof}

\begin{remark}
The conclusion of \cref{prop:sinexp} can also be written without Tail alternatives as 
\[
u(\cdot,t) \geq \eta M - \textup{Tail}_{\infty}(u_{-};x_0,\rho,(t_0,t_0+\delta\rho^{sp}M^{2-p})) \qquad \text{a.e. in } B_{2\rho}(x_0),
\]
for all times
\[
   t_0+(1-\varepsilon)\delta M^{2-p}\rho^{sp} \leq t \leq t_0+M^{2-p}\delta\rho^{sp},
\]
because we are working with $u$ nonnegative. 
\end{remark}
\begin{remark}
    In the proof of H\"older regularity we use \cref{prop:sinexp} with $\al = 1/2$ so that  $\eta, \delta$ and $\epsilon \in (0,1)$ depend only on the data in the proof. 
\end{remark}
\begin{remark}
In the degenerate case the final time level depends on the final lower level $\eta M$ that is achieved during the expansion whereas in the singular case, the final time level depends on the starting level $M$. 
\end{remark}

\section{H\"older regularity for degenerate parabolic fractional \texorpdfstring{$p$}.-Laplace equations}\label{sec9}

In this section, we present the proof of \cref{holderparabolic} in the case $p>2$. The induction argument is similar to the one in \cite{cozziRegularityResultsHarnack2017} and the covering argument is taken from \cite{hwangHolderContinuityBounded2015,hwangHolderContinuityBounded2015a}. Since the complete details are not easily available in the literature, we present all the calculations for the sake of completeness.
We assume without loss of generality that $(x_0,t_0)=(0,0)$ and choose $\beta\in (0,1)$ satisfying
\begin{equation}\label{choiceofalpha}
    0<\beta<\min\left\{ \frac{sp}{p-1},\frac{sp}{p-2},\log_{C_0}\left(\frac{2}{2-\eta}\right)\right\}\quad\mbox{ and }\quad
    \int_{1}^\infty \frac{(\rho^\beta-1)^{p-1}}{\rho^{1+sp}}\,d\rho<\frac{\eta^{p-1}}{2^p\mathbf{C}_1},
\end{equation} where $\eta$ is the constant that appears  in \cref{prop:degexp} for $\al = 1/2$  and $\mathbf{C}_1$ appears in \cref{estim4}. The constant $C_0\geq 2$ will be determined in a later step depending only on the data.
The integral can be made small enough by taking small $\beta$.
Also define
\begin{align}\label{choiceofj0}
    j_0:=\left\lceil\frac{1}{sp-\beta(p-1)}\log_{C_0}\left(\frac{2^p(\mathbf{C}_{3}+2^{p-1})}{\eta^{p-1}}\right)\right\rceil,
\end{align}where $\eta$ is again the constant that appears  in \cref{prop:degexp} for $\al = 1/2$.

\begin{claim}\label{claimfinal}
	For a universal constant $C_0\gg 2$ and  $b, \eta$ and $\delta$ as obtained in \cref{prop:degexp} for $\al = 1/2$, we claim that there exist non-decreasing sequence $\{m_i\}_{i=0}^\infty$ and non-increasing sequence $\{M_i\}_{i=0}^\infty$ such that for any $i=1,2,\ldots$, we have
	\begin{equation}\label{holdercond}
		m_i\leq u \leq M_i \qquad \mbox{ in } \ Q_{R_i}(d_i):=B_{R_i}\times (-d_i\,R_i^{sp},0) = B_i \times I_i.
	\end{equation} Here, we have denoted
	\begin{equation}\label{defofL}
		\begin{array}{rcl}
			M_i-m_i&=& C_0^{-\beta i}L, \\
			R_i &:= &C_0^{1-i}R,\\
		L &:=& 2\cdot C_0^{\frac{sp}{p-1}\,j_0}\|u\|_{L^\infty(B_{C_0R}\times (-(C_0R)^{2s},0))}+\tail(u;C_0\,R,0,(-(C_0\,R)^{2s},0)),\\
		d_i & := & \frac{b^{p-2}}{(\eta\,C_0^{-\beta i}L)^{p-2}}\delta\frac{2^{p-2}}{(2C_0)^{sp}}
		\end{array}
	\end{equation}

By a slight abuse of notation, only when $i=0$, we will denote $Q_0= Q_{R_0}(d_0) := B_{C_0R}\times (-(C_0R)^{2s},0)$ instead of $B_{C_0R}\times (-(C_0R)^{sp},0)$
\end{claim}
\begin{proof}[Proof of \cref{claimfinal}]
	The proof will proceed in two steps, first we show \cref{holdercond} holds for $i=0,1,2,\ldots,j_0$ and then use induction to obtain \cref{holdercond} for all $i>j_0$.
	\begin{description}
		\descitem{Step 1}{step11}
		Without loss of generality, we can assume $Q_{R_1}(d_1) \subset B_{C_0R}\times (-(C_0R)^{2s},0)$, since otherwise, we would have 
		\[
		L^{p-2} \leq R^{s(p-2)} \frac{b^{p-2}\delta}{C_0^{2s}(\eta C_0^{-\beta})^{p-2}}\frac{2^{p-2}}{(2C_0)^{sp}},
		\]
		which implies oscillation is comparable to the radius. 
		
		It is also easy to see that $Q_{R_{i+1}}(d_{i+1}) \subseteq Q_{R_i}(d_i)$ for $i = 1,2,\ldots$, since from \cref{choiceofalpha}, we have $\beta (p-2)-sp < 0$ which implies $d_{i+1} R_{i+1}^{sp} \leq d_iR_i^{sp}$. 
		\descitem{Step 2}{step111}
		 For $i=0,1,\ldots,j_0$, let us define $m_i:=\tfrac{-C_0^{-\beta i}L}{2}$ and $M_i:=\tfrac{C_0^{-\beta i}L}{2}$.
		 From \descref{step11}{Step 1}, we have $Q_{R_i}(d_i)\subset Q_{0}$ holds and thus we have
		\begin{equation*}
			\begin{array}{rcl}
				\|u\|_{L^\infty(Q_{R_i}(d_i))}\leq \|u\|_{L^\infty(Q_{0})}=\tfrac{C_0^{-\beta i}}{2}\left(2C_0^{\frac{sp}{p-1}j_0} \|u\|_{L^\infty(Q_{0})}\right)\left(\frac{C_0^{\beta i}}{C_0^{\frac{sp}{p-1}j_0}}\right)\overset{\cref{choiceofalpha}}{\leq} M_i.
			\end{array}
		\end{equation*}
		\descitem{Step 3}{step22} For some $j\geq j_0$, we suppose that the sequence $M_i$ and $m_i$ have been defined for $i=1,2,\ldots,j$. Inductively, we will construct $m_{j+1}$ and $M_{j+1}$ so that \cref{holdercond} holds.
		Define the function $v:=u - \tfrac{(M_j+m_j)}{2}$, then
	   using  monotonicity of $M_j$ and $m_j$, we see that the following holds $$(M_j-m_j)+2m_0\leq M_j+m_j\leq 2M_0-(M_j-m_j).$$ Recalling $m_0 = -\tfrac{L}{2}$ and $M_0= \tfrac{L}{2}$ along with the choice $m_i:=\tfrac{-C_0^{-\beta i}L}{2}$ and $M_i:=\tfrac{C_0^{-\beta i}L}{2}$ gives
	   \begin{equation}\label{Mmest1}
	   	-(1-C_0^{-\beta\,j})L\leq M_j+m_j\leq (1-C_0^{-\beta\,j})L.
	   \end{equation}
   Thus, we have
   \begin{equation}\label{estim1}
   	\left(\tfrac{C_0^{-\beta j}L}{2}\pm v\right)_-^{p-1}=\left(\tfrac{C_0^{-\beta j}L}{2}\pm u \mp \tfrac{M_j+m_j}{2}\right)_-^{p-1}\overset{\cref{Mmest1}}{\leq} \left(|u|+\tfrac{L}{2}\right)^{p-1}\leq 2^{p-1}|u|^{p-1}+L^{p-1}.
   \end{equation}

\descitem{Step 4}{step224}
Recalling the notation from \cref{holdercond}, we estimate 
\begin{equation}\label{tailest1}
	\begin{array}{rcl}
	&&\hspace{-2cm}\tail\biggr(\big(\tfrac{C_0^{-\beta j}L}{2}\pm v\big)_{-};C_0^{1-j}R,0,I_j\biggr)^{p-1}\\
	&=& C_0^{(1-j)sp}R^{sp}\, \esssup\limits_{I_{j}}\Bigg[\int_{B_0\setminus B_{j}}\frac{\left(\tfrac{C_0^{-\beta j}L}{2}\pm v\right)_-^{p-1}}{|x|^{N+ps}}\,dx+\int_{\RR^N\setminus B_0}\frac{\left(\tfrac{C_0^{-\beta j}L}{2}\pm v\right)_-^{p-1}}{|x|^{N+ps}}\,dx\Bigg]\\
	&\stackrel{\cref{estim1}}{\leq}& C_0^{(1-j)sp}R^{sp}\,\Bigg[\esssup\limits_{I_{j}}\int_{B_0\setminus B_{j}}\frac{\left(\tfrac{C_0^{-\beta j}L}{2}\pm v\right)_-^{p-1}}{|x|^{N+ps}}\,dx\\
	&& +2^{p-1}\esssup\limits_{I_{0}}\int_{\RR^N\setminus B_{0}}\frac{|u|^{p-1}}{|x|^{N+ps}}\,dx+L^{p-1}\int_{\RR^N\setminus B_{0}}\frac{1}{|x|^{N+ps}}\,dx\Bigg]\\
	&\stackrel{\cref{defofL}}{\leq}& C_0^{(1-j)sp}R^{sp}\,\Bigg[\underbrace{\esssup\limits_{I_{j}}\int_{B_0\setminus B_{j}}\frac{\left(\tfrac{C_0^{-\beta j}L}{2}\pm v\right)_-^{p-1}}{|x|^{N+ps}}\,dx}_{:=J} +2^{p-1}\frac{L^{p-1}}{(C_0\,R)^{sp}} + \mathbf{C}_{3} \frac{L^{p-1}}{(C_0\,R)^{sp}}\Bigg].
	\end{array}
\end{equation}
To estimate $J$, for given $x \in B_0 \setminus B_j$,  there is $l\in \{0,1,2,\ldots,j-1\}$ such that $x\in B_{l-1}\setminus B_{l}$ so that by the monotonicity of the sequence $m_i$ and by the induction hypothesis, for a.e. $t\in I_j$,  we have:
\begin{equation}\label{estim2}
	\begin{minipage}{0.4\textwidth}
	$\begin{array}{rcl}
		v(x,t) & \geq & -(M_l-m_l) + \tfrac{M_j-m_j}{2}\\
		& = & -\left[C_0^{-\beta l}  - \tfrac{C_0^{-\beta j}}{2}\right]L\\
		& \geq & -\left[\left(\tfrac{|x|}{C_0R}\right)^\beta - \tfrac{C_0^{-\beta j}}{2}\right]L
	\end{array}$
	\end{minipage}
\begin{minipage}{0.4\textwidth}
	$\begin{array}{rcl}
	v(x,t) & \leq & (M_l-m_l) -\tfrac{M_j-m_j}{2}\\
	& = & \left[C_0^{-\beta l}  - \tfrac{C_0^{-\beta j}}{2}\right]L\\
	&\leq& \left[\left(\tfrac{|x|}{C_0R}\right)^\beta - \tfrac{C_0^{-\beta j}}{2}\right]L
	\end{array}$
\end{minipage}
\end{equation}
From \cref{estim2}, we get
\begin{align*}
	\left(\tfrac{C_0^{-\beta j}L}{2}\pm v(x,t)\right)_-^{p-1}\leq  \left[\left(\frac{|x|}{C_0R}\right)^{\beta}-C_0^{-\beta j}\right]^{p-1} L^{p-1}
\end{align*}
We use this to estimate 
\begin{equation}\label{estim4}
	\begin{array}{rcl}
		J&\leq &C_0^{-\beta j(p-1)}L^{p-1}\esssup_{I_j}\int_{\RR^N\setminus B_{C_0^{1-j}R}}\frac{\left(\left(\frac{|x|}{C_0^{1-j}R}\right)^\beta-1\right)^{p-1}}{|x|^{N+sp}}\,dx\\
		&=&\frac{C_0^{-\beta j(p-1)}L^{p-1}C_0^{(j-1)sp}}{R^{sp}}\int_{\RR^N\setminus B_{1}}\frac{\left(|y|^\beta-1\right)^{p-1}}{|y|^{N+sp}}\,dy\\
		&\leq& \mathbf{C}_{1}\frac{C_0^{-\beta j(p-1)}L^{p-1}C_0^{(j-1)sp}}{R^{sp}}\int_{1}^\infty\frac{\left(\rho^\beta-1\right)^{p-1}}{\rho^{1+sp}}\,d\rho.
	\end{array}
\end{equation}
Substituting \cref{estim4} into \cref{tailest1}, we obtain
\begin{multline}\label{taildecay}
	\tail\left(\left(\tfrac{C_0^{-\beta j}L}{2}\pm v\right)_{-};C_0^{1-j}R,0,I_j\right)^{p-1}\\
	\begin{array}{rcl}
	&\leq & C_0^{(1-j)sp}R^{sp}\,\left[\mathbf{C}_{1}\frac{C_0^{-\beta j(p-1)}L^{p-1}C_0^{(j-1)ps}}{R^{sp}}\int_{1}^\infty\frac{\left(\rho^\beta-1\right)^{p-1}}{\rho^{1+ps}}\,d\rho +2^{p-1}\frac{L^{p-1}}{(C_0\,R)^{sp}}+\mathbf{C}_{3}\frac{L^{p-1}}{(C_0\,R)^{sp}}\right]\\
	&\leq & C_0^{-\beta j(p-1)}L^{p-1}\left[ \mathbf{C}_{1} \int_{1}^\infty\frac{(\rho^\beta-1)^{p-1}}{\rho^{1+ps}}\,d\rho+  ( \mathbf{C}_{3}+2^{p-1})\,C_0^{-(sp-\beta(p-1))j} \right]\\
	&\overset{\redlabel{7.10a}{a}}{\leq}& \frac{C_0^{-\beta j(p-1)}L^{p-1} \eta^{p-1}}{2^{p-1}},
	\end{array}
\end{multline} where \redref{7.10a}{a} follows from \cref{choiceofalpha}, \cref{choiceofj0} and the fact that $j\geq j_0$.

\descitem{Step 5}{step225}
Now, one of the following two alternatives must hold:
\begin{align}
	\left|\{v(\cdot,s_j)\geq 0\}\cap B_{C_0^{-j}R/2}\right|\geq \frac 12 \left|B_{C_0^{-j}R/2}\right|, \label{alt1}\\
	\left|\{v(\cdot,s_j)\geq 0\}\cap B_{C_0^{-j}R/2}\right|< \frac 12 \left|B_{C_0^{-j}R/2}\right|,\label{alt2}
\end{align}
where we have set
\begin{align*}
	s_j:=
		-\frac{b^{p-2}}{(\eta\,C_0^{-\beta j}L)^{p-2}}\delta R_j^{sp}\frac{2^{p-2}}{(2C_0)^{sp}}.
\end{align*}

Without loss of generality, we assume that \cref{alt1} holds, noting that if \cref{alt2} holds, analogous conclusion follows with $v$ replaced by $-v$. 
\descitem{Step 6}{step226}
Now denote $w=\tfrac{C_0^{-\beta j}L}{2}+v$, then \cref{alt1} becomes
\begin{align*}
	\left|\left\{w(\cdot,s_j)\geq \tfrac{C_0^{-\beta j}L}{2}\right\}\cap B_{C_0^{-j}R/2}\right|\geq \frac 12 \left|B_{C_0^{-j}R/2}\right|.
\end{align*}
Moreover from the induction hypothesis, we see that $w\geq 0$ in $Q_{R_j}(d_j)$  because
\begin{align*}
	w = v + \tfrac{C_0^{-\beta j}L}{2}= u - \tfrac{M_j+m_j}{2} + \tfrac{C_0^{-\beta j}L}{2}\overset{\cref{defofL}}{=} u - \tfrac{M_j+m_j}{2} + \tfrac{M_j-m_j}{2}= u - m_j\geq 0.
\end{align*}
We see that since \cref{taildecay} holds, the Tail alternative in \cref{prop:degexp} is satisfied. Therefore, we get
\begin{align}\label{redofosc}
	w\geq  \frac{\eta C_0^{-\beta j}L}{2} \qquad \mbox{ a.e. }(x,t)\in Q_{R_{j+1}}(d_{j+1}),
\end{align}
provided $C_0 > \left(\tfrac{1}{1-\sigma'}\right)^{\frac{1}{\beta(2-p)+sp}}$ where $\sigma'\in(0,1)$ is as obtained in \cref{prop:degexp} for $\al = 1/2$. Recalling \cref{choiceofalpha} we enforce such a universal choice of $C_0 \gg 2$. 
We calculate
\begin{equation*}
	\begin{array}{rcl}
		u(x,t)&=&\tfrac{M_j+m_j}{2} + v=\tfrac{M_j+m_j}{2} + w - \tfrac{C_0^{-\beta j}L}{2}\\
		&\overset{\cref{redofosc}}{\geq}& \tfrac{M_j+m_j}{2} - \tfrac{C_0^{-\beta j}L}{2}(1-\eta)\\
		&=& M_j - \tfrac{M_j-m_j}{2} - \tfrac{C_0^{-\beta j}L}{2}(1-\eta)\\
		&\overset{\redlabel{ind}{a}}{=}& M_j - \tfrac{C_0^{-\beta j}L}{2}(2-\eta)\\
		&\overset{\cref{choiceofalpha}}{\geq}& M_j - C_0^{-\beta(j+1)}L,
	\end{array}
\end{equation*} where \redref{ind}{a} follows from the induction hypothesis.
Thus if we define $m_{j+1}:=M_j - C_0^{-\beta(j+1)}L$ and $M_{j+1}:=M_j$ then the inductive process in \descref{step22}{Step 3} is completed. 
	\end{description} This completes the proof of the claim.
\end{proof}

A consequence of \cref{claimfinal} is the following oscillation decay:
\begin{equation}\label{oscdecays}
    \essosc_{Q_{R_i}(d_i)}\, u:=\esssup_{Q_{R_i}(d_i)}\,u - \essinf_{Q_{R_i}(d_i)}\, u \leq C_0^{-\beta i}L.
\end{equation}

\begin{proof}[Proof of \cref{holderparabolic}]
	
	With notation as in \cref{claimfinal}, let us take $Q_0:=B_{C_0R}\times (-(C_0R)^{2s},0)$ and define $L$, $R_i$ and $d_i$ as \cref{defofL}. 
	
	Consider any two points 
Let $(x_1,t_1) ,(x_2,t_2)\in  B_R(0)\times (-d_1^{2-p}R^{sp},0)$ such that $x_1\neq x_2$ and $t_1\neq t_2$, then there exist non-negative integers $n$ and $m$ such that
\begin{equation}\label{holdest1}
    R_{n+1}<|x_1-x_2|\leq R_n, \qquad \text{and} \qquad d_{m+1}R_{m+1}^{sp}<|t_1-t_2|\leq d_{m}R_{m}^{sp}.
\end{equation}
%
As a result, with $k := \max\{n,m\}$, we obtain
\begin{align}\label{holdest0}
    |u(x_1,t_1)-u(x_2,t_2)|\leq \essosc_{Q_{r_k}(d_k)}u \overset{\cref{{oscdecays}}}{\leq} \max\{C_0^{-\beta n},C_0^{-\beta m}\}L.
\end{align}

From the first inequality in \cref{holdest1}, we deduce
\begin{equation}\label{holdest2.5}
    \frac{|x_1-x_2|}{R}>C_0^{-n} \qquad \Longrightarrow \qquad \left(\frac{|x_1-x_2|}{R}\right)^\beta\,L>C_0^{-\beta n}L.
\end{equation} 
On the other hand, from the second inequality in \cref{holdest1}, we get
\begin{equation}\label{holdest2.5t}
	\mathbf{C}_4\,L^{\frac{p-2}{sp}}\frac{|t_1-t_2|^{\frac{1}{sp}}}{R}>\left(C_0^{\beta(m+1)}\right)^{\frac{p-2}{sp}}\,C_{0}^{-m} \geq C_0^{-m}
\end{equation} 
where 
\begin{equation}\label{eq:C4}
\mathbf{C}_4 = \lbr \frac{\eta^{p-2}}{b^{p-2}\delta}\frac{(2C_0)^{sp}}{2^{p-2}}\rbr^{\frac{1}{sp}}    
\end{equation}
is a constant depending only on data and the last inequality in \cref{holdest2.5t} follows since $p>2$ which implies the following $\left(C_0^{-\beta(m+1)}\right)^{\frac{2-p}{sp}}\geq 1$ holds as $C_0 > 1$. 

%
%
The proof of the theorem follows from substituting  \cref{holdest2.5} and  \cref{holdest2.5t} into \cref{holdest0}.
\end{proof}


\section{H\"older regularity for singular parabolic fractional \texorpdfstring{$p$}.-Laplace equations}\label{sec10}

In this section, we present the proof of \cref{holderparabolicsin} in the case $p<2$. The induction argument is similar to the degenerate proof from \cref{sec9}. Since the complete details are not easily available in the literature, we present all the calculations for the sake of completeness.
We assume without loss of generality that $(x_0,t_0)=(0,0)$ and choose $\beta\in (0,1)$ satisfying
\begin{equation}\label{choiceofalphasin}
	0<\beta<\min\left\{sp,\frac{sp}{2-p},\log_{C_0}\left(\frac{2}{2-\eta}\right)\right\}\quad\mbox{ and }\quad
	\int_{1}^\infty \frac{(\rho^\frac{\beta}{\Gamma}-1)^{p-1}}{\rho^{1+sp}}\,d\rho<\frac{\eta^{p-1}}{2^p\mathbf{C}_1},
\end{equation} where $\Gamma := \tfrac{\beta(p-2)}{sp}+1 > 0$, $\eta$ is the constant that appears  in \cref{prop:sinexp} for $\al = 1/2$ and $\mathbf{C}_1$ appears in \cref{estim4sin}. The constant $C_0 \geq 2$ will be determined in a later step depending only on the data
The integral can be made small enough by taking small $\beta$.
We further fix $\ve_0 > 0$ and
define
\begin{align}\label{choiceofj0sin}
	j_0:=\left\lceil\frac{1}{(sp-\beta)}\log_{C_0}\left(\frac{2^pC_0^{\ve_0sp}(\mathbf{C}_{3}+2^{p-1})}{\eta}\right)\right\rceil,
\end{align}where $\eta$ is the constant that appears  in \cref{prop:sinexp} for $\al = 1/2$. 

\begin{claim}\label{claimfinalsin}
	Let $R \leq 1$ be fixed, then for a universal constant $C_0\gg 2$ and  $b, \eta$ as obtained in \cref{prop:sinexp}, we claim that there exist non-decreasing sequence $\{m_i\}_{i=0}^\infty$ and non-increasing sequence $\{M_i\}_{i=0}^\infty$ such that for any $i=1,2,\ldots$, we have
	\begin{equation}\label{holdercondsin}
		m_i\leq u \leq M_i \qquad \mbox{ in } \ Q_{R_i}(d_i):=B_{d_iR_i}\times (-R_i^{sp},0) = B_i \times I_i.
	\end{equation} Here, we have denoted
	\begin{equation}\label{defofLsin}
		\begin{array}{rcl}
			M_i-m_i&=& C_0^{-\beta  i}L \overset{\cref{choiceofalphasin}}{\searrow} 0, \\
			R_i &:= &C_0^{1-i}R,\\
			L &:=& 2\cdot C_0^{-\beta  j_0}\|u\|_{L^\infty(B_{(C_0R)^{1-\ve_o}}\times (-(C_0R)^{sp},0))}+\tail(u;(C_0\,R)^{1-\ve_o},0,(-(C_0\,R)^{sp},0))+1,\\
			d_i & := & (\eta\,C_0^{-\beta i}L)^{\frac{p-2}{sp}}\delta
		\end{array}
	\end{equation}
	where $\delta>0$, $b >0$ and $\eta \in (0,1)$ appears in \cref{prop:sinexp} for $p<2$.
	
	By a slight abuse of notation, only when $i=0$, we will denote $Q_0= Q_{R_0}(d_0) := B_{(C_0R)^{1-\ve_o}}\times (-(C_0R)^{sp},0) = B_0\times I_0$.
\end{claim}
\begin{proof}[Proof of \cref{claimfinalsin}]
	The proof will proceed in two steps, first we show \cref{holdercondsin} holds for $i=0,1,2,\ldots,j_0$ and then use induction to obtain \cref{holdercondsin} for all $i>j_0$.
	\begin{description}
		\descitem{Step 1}{step11sin}
		Without loss of generality, we can assume $Q_{R_1}(d_1) \subset B_{(C_0R)^{1-\ve_o}}\times (-(C_0R)^{sp},0)$, since otherwise, we would have 
		\[
		L^{\frac{2-p}{sp}} \leq R^{\ve_o} \frac{(\eta C_0^{-\beta})^{\frac{p-2}{sp}}}{C_0^{1-\ve_o}},
		\]
		which implies oscillation is comparable to the radius. 
		
		It is also easy to see that $Q_{R_{i+1}}(d_{i+1}) \subseteq Q_{R_i}(d_i)$ for $i = 1,2,\ldots$, since from \cref{choiceofalphasin}, we have $\beta (2-p)-sp < 0$ which implies $d_{i+1} R_{i+1} \leq d_iR_i$. 
		\descitem{Step 2}{step111sin}
		For $i=0,1,\ldots,j_0$, let us define $m_i:=\tfrac{-C_0^{-\beta  i}L}{2}$ and $M_i:=\tfrac{C_0^{-\beta  i}L}{2}$.
		From \descref{step11sin}{Step 1}, we have $Q_{R_i}(d_i)\subset Q_{0}$ holds and thus we have
		\begin{equation*}
			\begin{array}{rcl}
				\|u\|_{L^\infty(Q_{R_i}(d_i))}\leq \|u\|_{L^\infty(Q_{0})}=\tfrac{C_0^{-\beta  i}}{2}\left(2C_0^{-\beta  j_0} \|u\|_{L^\infty(Q_{0})}\right)\left(\frac{C_0^{\beta  i}}{C_0^{-\beta  j_0}}\right)\overset{\cref{choiceofalphasin}}{\leq} M_i.
			\end{array}
		\end{equation*}
		\descitem{Step 3}{step22sin} For some $j\geq j_0$, we suppose that the sequence $M_i$ and $m_i$ have been defined for $i=1,2,\ldots,j$. Inductively, we will construct $m_{j+1}$ and $M_{j+1}$ so that \cref{holdercondsin} holds.
		Define the function $v:=u - \tfrac{(M_j+m_j)}{2}$, then
		using  monotonicity of $M_j$ and $m_j$, we see that the following holds $$(M_j-m_j)+2m_0\leq M_j+m_j\leq 2M_0-(M_j-m_j).$$ Recalling $m_0 = -\tfrac{L}{2}$ and $M_0= \tfrac{L}{2}$ along with the choice $m_i:=\tfrac{-C_0^{-\beta  i}L}{2}$ and $M_i:=\tfrac{C_0^{-\beta  i}L}{2}$ gives
		\begin{equation}\label{Mmest1sin}
			-(1-C_0^{-\beta  j})L\leq M_j+m_j\leq (1-C_0^{-\beta  j})L.
		\end{equation}
		Thus, we have
		\begin{equation}\label{estim1sin}
			\left(\tfrac{C_0^{-\beta  j}L}{2}\pm v\right)_-^{p-1}=\left(\tfrac{C_0^{-\beta  j}L}{2}\pm u \mp \tfrac{M_j+m_j}{2}\right)_-^{p-1}\overset{\cref{Mmest1sin}}{\leq} \left(|u|+\tfrac{L}{2}\right)^{p-1}\leq 2^{p-1}|u|^{p-1}+L^{p-1}.
		\end{equation}
		\descitem{Step 4}{step224sin}
		Recalling the notation from \cref{holdercondsin}, we estimate 
		\begin{equation}\label{tailest1sin}
			\begin{array}{rcl}
				&&\hspace{-2cm}\tail\biggr(\big(\tfrac{C_0^{-\beta  j}L}{2}\pm v\big)_{-};B_j,0,I_j\biggr)^{p-1}\\
				&=& (d_jR_j)^{sp}\, \esssup\limits_{I_{j}}\Bigg[\int_{B_0\setminus B_{j}}\frac{\left(\tfrac{C_0^{-\beta  j}L}{2}\pm v\right)_-^{p-1}}{|x|^{N+ps}}\,dx+\int_{\RR^N\setminus B_0}\frac{\left(\tfrac{C_0^{-\beta  j}L}{2}\pm v\right)_-^{p-1}}{|x|^{N+ps}}\,dx\Bigg]\\
				&\stackrel{\cref{estim1sin}}{\leq}& (d_jR_j)^{sp}\,\Bigg[\esssup\limits_{I_{j}}\int_{B_0\setminus B_{j}}\frac{\left(\tfrac{C_0^{-\beta  j}L}{2}\pm v\right)_-^{p-1}}{|x|^{N+ps}}\,dx\\
				&& +2^{p-1}\esssup\limits_{I_{0}}\int_{\RR^N\setminus B_{0}}\frac{|u|^{p-1}}{|x|^{N+ps}}\,dx+L^{p-1}\int_{\RR^N\setminus B_{0}}\frac{1}{|x|^{N+ps}}\,dx\Bigg]\\
				&\stackrel{\cref{defofLsin}}{\leq}& (d_jR_j)^{sp}\,\Bigg[\underbrace{\esssup\limits_{I_{j}}\int_{B_0\setminus B_{j}}\frac{\left(\tfrac{C_0^{-\beta  j}L}{2}\pm v\right)_-^{p-1}}{|x|^{N+ps}}\,dx}_{:=J} +2^{p-1}\frac{L^{p-1}}{(C_0\,R)^{sp(1-\ve_0)}} \\
				&&+ \mathbf{C}_{3} \frac{L^{p-1}}{(C_0\,R)^{sp(1-\ve_0)}}\Bigg].
			\end{array}
		\end{equation}
		To estimate $J$, for given $x \in B_0 \setminus B_j$,  there is $l\in \{0,1,2,\ldots,j-1\}$ such that $x\in B_{l-1}\setminus B_{l}$ which implies $|x| \geq d_{l}R_{l}$ which is equivalent to the condition $\frac{|x|}{C_0R} \frac{(\eta L)^{\frac{2-p}{sp}}}{\delta} \geq C_0^{-l\Gamma}$
		where $\Gamma = \tfrac{\beta(p-2)}{sp}+1 \geq 0$ (by \cref{choiceofalphasin})
		so that by the monotonicity of the sequence $m_i$ and by the induction hypothesis, for a.e. $t\in I_j$,  we have:
		\begin{equation}\label{estim2sin}
				\begin{array}{rcl}
					v(x,t) & \geq & -(M_l-m_l) + \tfrac{M_j-m_j}{2}\\
					& = & -\left[C_0^{-\beta  l}  - \tfrac{C_0^{-\beta  j}}{2}\right]L\\
					& \geq & -\left[\left(\frac{|x|}{C_0R} \frac{(\eta L)^{\frac{2-p}{sp}}}{\delta}\right)^{\frac{\beta}{\Gamma}} - \tfrac{C_0^{-\beta  j}}{2}\right]L
				\end{array}
			\end{equation}
and 
\begin{equation}\label{estim3sin}
				\begin{array}{rcl}
					v(x,t) & \leq & (M_l-m_l) - \tfrac{M_j-m_j}{2}\\
					& = & \left[C_0^{-\beta  l}  - \tfrac{C_0^{-\beta  j}}{2}\right]L\\
					& \leq & \left[\left(\frac{|x|}{C_0R} \frac{(\eta L)^{\frac{2-p}{sp}}}{\delta}\right)^{\frac{\beta}{\Gamma}} - \tfrac{C_0^{-\beta  j}}{2}\right]L
				\end{array}
		\end{equation}
		From \cref{estim2sin} and \cref{estim3sin}, we get
		\begin{align*}
			\left(\tfrac{C_0^{-\beta  j}L}{2}\pm v(x,t)\right)_-^{p-1}\leq  \left[\left(\frac{|x|}{C_0R} \frac{(\eta L)^{\frac{2-p}{sp}}}{\delta}\right)^{\frac{\beta}{\Gamma}} - C_0^{-\beta  j}\right]^{p-1} L^{p-1}
		\end{align*}
		We use this and the definition of $d_jR_j$ from \cref{defofLsin} to estimate 
		\begin{equation}\label{estim4sin}
			\begin{array}{rcl}
				J&\leq &C_0^{-\beta  j(p-1)}L^{p-1}\esssup_{I_j}\int_{\RR^N\setminus B_{j}}\frac{\left(\left(\frac{|x|}{d_lR_j}\right)^{\frac{\beta}{\Gamma}}-1\right)^{p-1}}{|x|^{N+sp}}\,dx\\
				&=&\frac{C_0^{-\beta j(p-1)}L^{p-1}}{(d_jR_j)^{sp}} \int_{\RR^N\setminus B_{1}}\frac{\left(|y|^{\frac{\beta}{\Gamma}}-1\right)^{p-1}}{|y|^{N+sp}}\,dy\\
				&=& \mathbf{C}_{1}\frac{C_0^{-\beta j(p-1)}L^{p-1}}{(d_jR_j)^{sp}} \int_{1}^{\infty}\frac{\left(\rho^{\frac{\beta}{\Gamma}}-1\right)^{p-1}}{{\rho}^{1+sp}}\,dy\rho\\
			\end{array}
		\end{equation}
		Substituting \cref{estim4sin} into \cref{tailest1sin}, we obtain
		\begin{multline}\label{taildecaysin}
			\tail\left(\left(\tfrac{C_0^{-\beta j}L}{2}\pm v\right)_{-};B_j,0,I_j\right)^{p-1}\\
			\begin{array}{rcl}
				&\leq & (d_jR_j)^{sp}\,\left[\mathbf{C}_{1}\frac{C_0^{-\beta j(p-1)}L^{p-1}}{(d_jR_j)^{sp}} \int_{1}^{\infty}\frac{\left(\rho^{\frac{\beta}{\Gamma}}-1\right)^{p-1}}{{\rho}^{1+sp}}\,d\rho +2^{p-1}\frac{L^{p-1}}{(C_0\,R)^{sp(1-\ve_0)}}+\mathbf{C}_{3}\frac{L^{p-1}}{(C_0\,R)^{sp(1-\ve_0)}}\right]\\
				&\leq & \mathbf{C}_{1}C_0^{-\beta j(p-1)}L^{p-1} \int_{1}^{\infty}\frac{\left(\rho^{\frac{\beta}{\Gamma}}-1\right)^{p-1}}{{\rho}^{1+sp}}\,d\rho+(2^{p-1}+\mathbf{C}_{3})(d_jR_j)^{sp}\frac{L^{p-1}}{(C_0\,R)^{sp(1-\ve_0)}}\\
				&\overset{\redlabel{7.10asin}{a}}{\leq}& \frac{C_0^{-\beta j(p-1)}L^{p-1} \eta^{p-1}}{2^{p-1}},
			\end{array}
		\end{multline} where \redref{7.10asin}{a} follows from \cref{choiceofalphasin},\cref{choiceofj0sin},\cref{defofLsin} and the fact that $j\geq j_0$, $L \geq 1$ and $R,\delta \leq 1$. Indeed, from \cref{choiceofalphasin} we have
  \[
  \mathbf{C}_{1}C_0^{-\beta j(p-1)}L^{p-1} \int_{1}^{\infty}\frac{\left(\rho^{\frac{\beta}{\Gamma}}-1\right)^{p-1}}{{\rho}^{1+sp}}\,d\rho \leq \frac{C_0^{-\beta j(p-1)}L^{p-1} \eta^{p-1}}{2^{p}},
  \]
  and 
\[
(2^{p-1}+\mathbf{C}_{3})(d_jR_j)^{sp}\frac{L^{p-1}}{(C_0\,R)^{sp(1-\ve_0)}} \leq (2^{p-1}+\mathbf{C}_{3})\eta^{p-2}C_0^{\ve_0sp}C_0^{j(\beta-sp)}L^{p-1}C_0^{-\beta j(p-1)}
\]
where we used the definition of $d_jR_j$ from \cref{defofLsin}, $\delta^{sp} \leq 1$, $R^{\ve_0} \leq 1$ and $L^{p-2} \leq 1$ as $L \geq 1$ and $p<2$. Using \cref{choiceofj0sin} and that $j \geq j_0$ we get
\[
C_0^{j(\beta-sp)} \leq \left(\frac{2^pC_0^{\ve_0sp}(\mathbf{C}_{3}+2^{p-1})}{\eta}\right)^{-1}
\]
so that 
\[
(2^{p-1}+\mathbf{C}_{3})(d_jR_j)^{sp}\frac{L^{p-1}}{(C_0\,R)^{sp(1-\ve_0)}} \leq \frac{C_0^{-\beta j(p-1)}L^{p-1} \eta^{p-1}}{2^{p}},
\]  
		
		\descitem{Step 5}{step225sin}
		Now, one of the following two alternatives must hold:
		\begin{align}
			\left|\{v(\cdot,s_j)\geq 0\}\cap \tfrac{1}{2}B_{j+1}\right|\geq \frac 12 \left|\tfrac{1}{2}B_{j+1}\right|, \label{alt1sin}\\
			\left|\{v(\cdot,s_j)\geq 0\}\cap \tfrac{1}{2}B_{j+1}\right|< \frac 12 \left|\tfrac{1}{2}B_{j+1}\right|,\label{alt2sin}
		\end{align}
		where we have set
		\begin{align*}
			s_j:=
			-\left(\frac{C_0^{-\beta j }L}{2}\right)^{2-p}\left(\frac{d_{j+1}R_{j+1}}{2}\right)^{sp}\delta.
		\end{align*}

		Without loss of generality, we assume that \cref{alt1sin} holds, noting that if \cref{alt2sin} holds, analogous conclusion follows with $v$ replaced by $-v$. Before proceeding, let us ensure that the time level $s_j$ is in the interval $(-R_j^{sp},0)$ so that we can apply \cref{prop:sinexp} in the next step. Indeed, we want
  \[
  -R_j^{sp} \leq -s_j
  \]
  which is equivalent to 
  \[
  R_j^{sp} \geq \left(\frac{C_0^{-\beta j }L}{2}\right)^{2-p}\left(\frac{d_{j+1}R_{j+1}}{2}\right)^{sp}\delta
  \]
  Plugging in the definition of $d_j$ from \cref{defofLsin} we get
  \[
 C_0^{sp+\beta(p-2)} \geq \frac{\delta^{1+sp}}{\eta^{2-p}2^{2-p}2^{sp}}.
  \]
  Since $sp+\beta(p-2) > 0$ we can enforce such a universal choice of $C_0 > 1$. 
		\descitem{Step 6}{step226sin}
		Now denote $w=\tfrac{C_0^{-\beta j}L}{2}+v$, then \cref{alt1sin} becomes
		\begin{align*}
			\left|\left\{w(\cdot,s_j)\geq \tfrac{C_0^{-\beta j}L}{2}\right\}\cap \tfrac{1}{2}B_{j+1}\right|\geq \frac 12 \left|\tfrac{1}{2}B_{j+1}\right|.
		\end{align*}
		Moreover from the induction hypothesis, we see that $w\geq 0$ in $Q_{R_j}(d_j)$  because
		\begin{align*}
			w = v + \tfrac{C_0^{-\beta j}L}{2}= u - \tfrac{M_j+m_j}{2} + \tfrac{C_0^{-\beta j}L}{2}\overset{\cref{defofLsin}}{=} u - \tfrac{M_j+m_j}{2} + \tfrac{M_j-m_j}{2}= u - m_j\geq 0.
		\end{align*}
		We see that since \cref{taildecaysin} holds, the Tail alternative in \cref{prop:sinexp} is satisfied. Therefore, we get
		\begin{align}\label{redofoscsin}
			w\geq  \frac{\eta C_0^{-\beta j}L}{2} \qquad \mbox{ a.e. }(x,t)\in Q_{R_{j+1}}(d_{j+1}),
		\end{align}
		which, by computations similar to  the previous step holds provided:
  \[
  C_0^{-\beta(p-2)} \geq 2^{(2-p)+sp}\frac{\eta^{2-p}}{\delta^{1+sp}\ve},
  \]
  where $\ve \in (0,1)$ is the as obtained in \cref{prop:sinexp} for $\al = 1/2$. Since $p<2$ we can further enforce such a universal choice of $C_0>1$. 

		We calculate
		\begin{equation*}
			\begin{array}{rcl}
				u(x,t)&=&\tfrac{M_j+m_j}{2} + v=\tfrac{M_j+m_j}{2} + w - \tfrac{C_0^{-\beta j}L}{2}\\
				&\overset{\cref{redofoscsin}}{\geq}& \tfrac{M_j+m_j}{2} - \tfrac{C_0^{-\beta j}L}{2}(1-\eta)\\
				&=& M_j - \tfrac{M_j-m_j}{2} - \tfrac{C_0^{-\beta j}L}{2}(1-\eta)\\
				&\overset{\redlabel{indsin}{a}}{=}& M_j - \tfrac{C_0^{-\beta j}L}{2}(2-\eta)\\
				&\overset{\cref{choiceofalphasin}}{\geq}& M_j - C_0^{-\beta(j+1)}L,
			\end{array}
		\end{equation*} where \redref{indsin}{a} follows from the induction hypothesis.
		Thus if we define $m_{j+1}:=M_j - C_0^{-\beta(j+1)}L$ and $M_{j+1}:=M_j$ then the inductive process in \descref{step22sin}{Step 3} is completed. 
	\end{description} This completes the proof of the claim.
\end{proof}

A consequence of \cref{claimfinalsin} is the following oscillation decay:
\begin{equation}\label{oscdecayssin}
	\essosc_{Q_{R_i}(d_i)}\, u:=\esssup_{Q_{R_i}(d_i)}\,u - \essinf_{Q_{R_i}(d_i)}\, u \leq C_0^{-\beta i}L.
\end{equation}

\begin{proof}[Proof of \cref{holderparabolicsin}]
	
	With notation as in \cref{claimfinalsin}, let us take $Q_0:=B_{C_0R}\times (-(C_0R)^{2s},0)$ and define $L$, $R_i$ and $d_i$ as \cref{defofLsin}. 
	
	Consider any two points 
	Let $(x_1,t_1) ,(x_2,t_2)\in  Q_1(d_1)$ such that $x_1\neq x_2$ and $t_1\neq t_2$, then there exist non-negative integers $n$ and $m$ such that
	\begin{equation}\label{holdest1sin}
		d_{n+1}R_{n+1}<|x_1-x_2|\leq d_nR_n, \qquad \text{and} \qquad R_{m+1}^{sp}<|t_1-t_2|\leq R_{m}^{sp}.
	\end{equation}
	%
	As a result, with $k := \max\{n,m\}$, we obtain
	\begin{align}\label{holdest0sin}
		|u(x_1,t_1)-u(x_2,t_2)|\leq \essosc_{Q_{r_k}(d_k)}u \overset{\cref{{oscdecayssin}}}{\leq} \max\{C_0^{-\beta n},C_0^{-\beta m}\}L.
	\end{align}
	
	From the first inequality in \cref{holdest1sin}, we deduce
	\begin{equation}\label{holdest2.5sin}
		\mathbf{C}_4L^{\frac{2-p}{sp}}\frac{|x_1-x_2|}{R}>C_0^{-n}
	\end{equation} 
 where 
 \begin{equation}\label{eq:C4sin}
     \mathbf{C}_4 = \frac{\eta^{\frac{2-p}{sp}}}{\delta}
 \end{equation}
	On the other hand, from the second inequality in \cref{holdest1sin}, we get
	\begin{equation}\label{holdest2.5tsin}
		\frac{|t_1-t_2|^{\frac{1}{sp}}}{R}> C_0^{-m}
	\end{equation}  
	
	%
	%
	The proof of the theorem follows from substituting  \cref{holdest2.5sin} and  \cref{holdest2.5tsin} into \cref{holdest0sin}.
\end{proof}

\begin{remark}
	We note that the proof provides the following relationship between the Holder exponents for space and time:
	\begin{center}
		\begin{tabular}{||c| c| c||} 
			\hline\hline
			& x & t \\ 
			\hline\hline 
			Degenerate & $\beta$ & $\tfrac{\beta}{sp}$\\
			\hline 
			Singular & $\beta$ & $\tfrac{\beta}{sp}$\\  
			\hline\hline
		\end{tabular}
	\end{center}
	In the degenerate case, the condition $\beta(p-1) < sp$ forces $\beta < sp$. 
\end{remark}

\begin{remark}\label{rem:nonlocalB}
	We further note that the space Holder exponent is $\beta$ while the time exponent is $\beta/sp$. In particular, if we start with $sp<1$ time regularity beats space regularity - such a thing does not occur in the local case and is an instance of a purely nonlocal phenomena. 
\end{remark}

\section*{Acknowledgements}

The authors thank Agnid Banerjee and Prashanta Garain for  helpful discussions.


\section*{References}

\end{document}